\documentclass[11pt]{amsart}
\usepackage{amsmath, amssymb, amsfonts, verbatim, graphicx, amsthm, tikz-cd, mathrsfs, color, comment, fancyhdr, amsrefs}
\usepackage[top=1in, bottom=1in, left=1in, right=1in]{geometry}
\usepackage{enumerate}

\usepackage{mathtools}

\usepackage{url}

\usepackage{enumitem}

\usepackage[symbol]{footmisc}

\usepackage[foot]{amsaddr}

\allowdisplaybreaks[1]

\usepackage[sc]{mathpazo}
\usepackage[T1]{fontenc}

\newcommand{\vertiii}[1]{{\left\vert\kern-0.25ex\left\vert\kern-0.25ex\left\vert #1
		\right\vert\kern-0.25ex\right\vert\kern-0.25ex\right\vert}}

\newcommand{\Folner}{Følner }

\theoremstyle{plain}
\newtheorem{Thm}{Theorem}[section]
\newtheorem{Prop}[Thm]{Proposition}
\newtheorem{Lem}[Thm]{Lemma}
\newtheorem{Cor}[Thm]{Corollary}

\newtheorem*{Main results}{Main results about the gauge}
\newtheorem*{W-W}{Wiener-Wintner pointwise ergodic theorem}

\theoremstyle{definition}
\newtheorem{Def}[Thm]{Definition}
\newtheorem{Not}[Thm]{Notation}
\newtheorem{Rmk}[Thm]{Remark}

\renewcommand{\epsilon}{\varepsilon}

\usepackage{setspace}

\usepackage{hyperref}

\title{Temporo-spatial differentiations for actions of locally compact groups}

\author{Aidan Young$^1$}
\address{University of North Carolina at Chapel Hill}

\email{$^1$\url{aidanjy@live.unc.edu}}
\begin{document}
	
	\maketitle
	
	\begin{abstract}
		In this paper, we extend the notion of temporo-spatial differentiation problems to the setting of actions of more general topological groups. The problem can be expressed as follows: Given an action $T$ of an amenable discrete group $G$ on a probability space $(X, \mu)$ by automorphisms, let $(F_k)_{k = 1}^\infty$ be a Følner sequence for $G$, and let $(C_k)_{k = 1}^\infty$ be a sequence of measurable subsets of $X$ with positive probability $\mu(C_k)$. What is the limiting behavior of the sequence
		$$\left( \frac{1}{\mu(C_k)} \int_{C_k} \frac{1}{|F_k|} \sum_{g \in F_k} f(T_g x) \mathrm{d} \mu(x) \right)_{k = 1}^\infty$$
		for $f \in L^\infty(X, \mu)$? We provide some positive convergence results for temporo-spatial differentiations with respect to ergodic averages over Følner sequences, as well as with respect to ergodic averages over subsequences of the integers (e.g. polynomials), multiple ergodic averages, and weighted ergodic averages.
	\end{abstract}
	
	Temporo-spatial differentiation problems were introduced under the name of ``spatial-temporal differentiation problems" in \cite{Assani-Young}\footnote[2]{Both authors of \cite{Assani-Young} felt that this new name was a bit easier to pronounce than the original.} in the context of actions of $\mathbb{Z}$. Here, we extend this concept to a more general setting. For our purposes, a temporo-spatial differentiation problem is a question of the following form: Let $T : G \curvearrowright (X, \mu)$ be a measure-preserving action of a locally compact group $G$ on a probability space $(X, \mu)$, and suppose $(F_i)_{i \in \mathscr{I}}$ is a net of compact subsets of $G$ with positive Haar measure $m(F_i)$. Let $(C_i)_{i \in \mathscr{I}}$ be a net of measurable subsets of $X$ with positive measure, and $f \in L^\infty(X, \mu)$. What can be said about the limiting behavior of the net
	$$\left( \frac{1}{\mu(C_i)} \int_{C_i} \frac{1}{m(F_i)} \int_{F_i} f(T_g x) \mathrm{d} m(g) \mathrm{d} \mu(x) \right)_{i \in \mathscr{I}} ?$$
	Though we will consider temporo-spatial differentiation problems that don't fit exactly into this mold, this description captures the problem in its basic form, including its distinctive features: we consider a limit of averages with both a a temporal aspect (provided by the average over $F_i$) and a spatial aspect (provided by the average over $C_i$).
	
	Our primary goal with this paper is to take the material originally presented in \cite{Assani-Young} and extend it by showing that certain concepts from that earlier paper can be extended to various different interesting settings.
	
	In Section \ref{Topological stuff}, we provide some general results about temporo-spatial differentiations. In particular, we provide a characterization in terms of ergodic optimization of a kind of ``best-case scenario" behavior, where temporo-spatial averages of continuous functions always converge to the integral.
	
	In Section \ref{Large averaging sets and static averaging sets}, we provide convergence theorems for two special cases of temporo-spatial differentiation averages: where the spatial averaging sets have measure going to $1$, and where they are constant.
	
	In Section \ref{Temporo-spatial differentiation theorems around sets of rapidly vanishing diameter}, we show that in the case where the spatial averaging sets share a common fixed point $x$ and have diameter going to $0$ sufficiently fast, then the associated temporo-spatial differentiations can be reduced to a pointwise temporal average at that fixed point $x$. This then provides us a means to prove convergence results for suitable ``random temporo-spatial differentiation problems." In particular, these reduction results can be applied even when the temporal averaging sets are not \Folner.
	
	In Section \ref{Weighed temporo-spatial differentiation theorems}, we generalize some of the results of Section \ref{Temporo-spatial differentiation theorems around sets of rapidly vanishing diameter} to the setting of weighted temporo-spatial ergodic averages. These include equicontinuous families of continuous weight functions of modulus 1, as well as potentially unbounded weight sequences of complex constants.
	
	\section{General results and unique ergodicity}\label{Topological stuff}
	
	Throughout this article, by a \emph{topological dynamical system}, we will mean a continuous action $T$ of a locally compact \emph{bimodular} topological group $G$ on a compact metrizable space $X$, denoted $T : G \curvearrowright X$. We will use $m$ to refer to a left- and right-invariant Haar measure on the group $G$, hereafter referred to simply as a Haar measure. Our consideration of bimodular groups is primarily to simplify some bookkeeping about when we are invoking a left-invariant Haar measure and when we are invoking a right-invariant Haar measure. Of course, the class of bimodular groups includes all abelian groups, all discrete groups, and all compact groups, thus encompassing many of the groups ergodic theory classically considers actions of.
	
	\begin{Def}
		Let $T : G \curvearrowright X$ be a topological dynamical system, where $G$ is an amenable group, and let $f \in C_{\mathbb{R}}(X)$ be a real-valued continuous function on $X$. The \emph{gauge} of $f$ is the value
		$$\Gamma(f) : = \sup \left\{ \int f \mathrm{d} \mu : \mu \in \mathcal{M}_T (X) \right\} , $$
		where $\mathcal{M}_T(X)$ denotes the family of $T$-invariant Borel probability measures on $X$. We say that $\mu \in \mathcal{M}_T(X)$ is \emph{$f$-maximizing} if $\int f \mathrm{d} \mu = \Gamma(f)$, and denote the class of all $f$-maximizing measures on $X$ by
		$$\mathcal{M}_{\mathrm{max}}(f) : = \left\{ \mu \in \mathcal{M}_T(X) : \int f \mathrm{d} \mu = \Gamma(f) \right\} .$$
	\end{Def}
	
	The gauge is well-defined, since if $G$ is amenable, then $\mathcal{M}_T(X)$ is a nonempty Choquet simplex in the weak*-topology. Since $\mathcal{M}_T (X)$ is compact, it follows that $\mathcal{M}_{\mathrm{max}}(f)$ is nonempty.
	
	We now wish to provide an alternative description of the gauge for nonnegative-valued functions.
	
	\begin{Lem}\label{Averages are continuous}
		Let $T : G \curvearrowright X$ be a topological dynamical system, and let $K \subseteq G$ be a compact subset of a locally compact group $G$. Let $f \in C(X)$. Then the function $x \mapsto \int_K f(T_g x) \mathrm{d} m(g)$ is continuous, where $m$ is a Haar measure on $G$.
	\end{Lem}
	
	\begin{proof}
		We can assume that $K$ is of positive Haar measure, and in particular nonempty, since otherwise this would be trivial.
		
		Fix $\epsilon > 0$, and let $\rho$ be a compatible metric for $X$. We know a priori that the function $G \times X \to \mathbb{C}$ given by $(g, x) \mapsto f(T_g x)$ is continuous, so for each $g \in K$, choose an open neighborhood $U_g \subseteq G$ of $g$ and a positive number $\delta_{g} > 0$ such that if $\left( g' , x' \right) \in U_{g} \times B \left( x, \delta_{g} \right)$, then $\left| f \left(T_{g'} x' \right) - f(T_g x) \right| < \frac{\epsilon}{2 m(K)}$. Then $\left\{U_g\right\}_{g \in K}$ is an open cover of the compact $K$, so there exist $g_1, \ldots, g_n \in K$ such that $K \subseteq U_{g_1} \cup \cdots \cup U_{g_n}$. Let $\delta = \min \left\{ \delta_{g_1}, \ldots, \delta_{g_n} \right\}$. Then if $\rho(x, y) < \delta$, and $g \in K$, then $g \in U_{g_j}$ for some $j \in \{1, \ldots, n\}$. Therefore $(g, x), (g, y) \in U_{g_j} \times B(x, \delta) \subseteq U_{g_j} \times B (x, \delta_j)$, so
		$$|f(T_g x) - f(T_g y)| \leq \left| f(T_g x) - f\left(T_{g_j} x\right) \right| + \left| f \left( T_{g_j} x \right) - f \left( T_g y \right) \right| < \frac{\epsilon}{2 m(K)} + \frac{\epsilon}{2 m(K)} = \frac{\epsilon}{m(K)} .$$
		
		Thus there exists $\delta > 0$ such that if $\rho(x, y) < \delta$, then $\left| f(T_g x) - f(T_g y) \right| < \frac{\epsilon}{m(K)}$. Therefore, it follows that if $\rho(x, y) < \delta$, then
		\begin{align*}
			\left| \int_K f(T_g x) \mathrm{d} m(g) - \int_K f(T_g y) \mathrm{d} m(g) \right|	& = \left| \int_K \left( f(T_g x) - f(T_g y) \right) \mathrm{d} m(g) \right| \\
			& \leq \int_K \left| f(T_g x) - f(T_g y) \right| \mathrm{d} m(g) \\
			& \leq \int_K \frac{\epsilon}{m(K)} \mathrm{d} m(g) \\
			& = \epsilon .
		\end{align*}
		Therefore the function $x \mapsto \int_k f(T_g x) \mathrm{d} m(g)$ is continuous.
	\end{proof}
	
	\begin{Not}
		\begin{enumerate}[label=(\alph*)]
			\item Let $T : G \curvearrowright X$ be a continuous action of a locally compact group $G$ with Haar measure $m$ on a topological space $X$, and let $f$ be a continuous function $X \to \mathbb{C}$. Let $K$ be a compact subset of $G$ with positive Haar measure. We define $\operatorname{Avg}_K f : X \to \mathbb{C}$ to be the continuous function
			$$\operatorname{Avg}_K f (x) = \frac{1}{m(K)} \int_K f(T_g x) \mathrm{d} m(g) .$$
			The continuity of $\operatorname{Avg}_K f$ follows from Lemma \ref{Averages are continuous}. In the event where $G$ is a discrete group, we will also define $\operatorname{Avg}_K f$ for all nonempty compact subsets $K$ of $G$ and $f \in L^1(X, \mu)$ by
			$$\operatorname{Avg}_K f = \frac{1}{|K|} \sum_{g \in K} f \circ T_g .$$
			\item Let $(X, \mu)$ be a probability space, and let $f \in L^1(X, \mu)$. Let $C$ be a measurable subset of $X$ with $\mu(C) > 0$. We define the functional $\alpha_C : L^1(X, \mu) \to \mathbb{C}$ by
			$$\alpha_C(f) : = \frac{1}{\mu(C)} . $$
		\end{enumerate}
		Although the functionals $\alpha_C$ are defined here on $L^1$, we will almost always be interested in their action on $L^\infty$, where they are considerably better behaved.
	\end{Not}
	
	\begin{Def}\label{Chapter 3 definition of Folner}
		Let $G$ be a locally compact topological group. A net $(F_i)_{i \in \mathscr{I}}$ of compact subsets of $G$ is called \emph{\Folner} if $m(F_i) > 0$ for all $i \in \mathscr{I}$, and
		\begin{align*}
			\lim_i \frac{m(g F_i \Delta F_i)}{m(F_i)}	& = 0	& (\forall g \in G),
		\end{align*}
		where $\Delta$ denotes the symmetric difference $A \Delta B = (A \setminus B) \cup (B \setminus A)$.
	\end{Def}
	
	\begin{Thm}\label{Gauge converges}
		Let $T : G \curvearrowright X$ be a topological dynamical system, where $G$ is an amenable group with Haar measure $m$. Let $(F_i)_{i \in \mathscr{I}}$ be a left \Folner \space net for $G$, and let $f \in C_\mathbb{R}(X)$ be a nonnegative-valued continuous function on $X$. Then the net $\left( \left\| \operatorname{Avg}_{F_i} f \right\|_{C(X)} \right)_{i \in \mathscr{I}}$ converges, and
		$$\Gamma(f) = \lim_i \left\| \operatorname{Avg}_{F_i} f \right\| .$$
	\end{Thm}
	
	\begin{proof}
		For each $i \in \mathscr{I}$, let $\sigma_i$ be a Borel probability measure on $X$ such that 
		$$\int \operatorname{Avg}_{F_i} f \mathrm{d} \sigma_i = \left\| \operatorname{Avg}_{F_i} f \right\|_{C(X)} .$$ 
		For each $i \in \mathscr{I}$, define the Borel probability measure $\mu_i$ on $X$ by
		$$\int f \mathrm{d} \mu_i = \int_X \left( \operatorname{Avg}_{F_i} f \right) \mathrm{d} \sigma_i = \frac{1}{m(F_i)} \int_{F_i} \left( \int_X T_g f(x) \mathrm{d} \sigma_i(x) \right) \mathrm{d} m(g) ,$$
		where the latter equality follows from Fubini's Theorem.
		
		In order to prove that the net $\left( \int f \mathrm{d} \mu_i \right)_{i \in \mathscr{I}}$ converges to $\Gamma(f)$, it will suffice to prove that for any convergent sub-net $\left( \mu_{i_j} \right)_{j \in \mathscr{J}}$, the sub-net $\left( \int f \mathrm{d} \mu_{i_j} \right)_{j \in \mathscr{J}}$ converges to $\Gamma(f)$, because if $\left( \int f \mathrm{d} \mu_{i} \right)_{i \in \mathscr{I}}$ didn't converge to $\Gamma(f)$, then we could extract some subnet $\left( \mu_{i_j} \right)_{j \in \mathscr{J}}$ along which $\left( \int f \mathrm{d} \mu_{i_j} \right)_{j \in \mathscr{J}}$ converged to some other point (because the net is contained in a compact subset of $\mathbb{R}$), then take a weak*-convergent subnet of that, yielding a contradiction.
		
		Therefore, we need to show that every weak*-limit point of the net $\left( \mu_i \right)_{i \in \mathscr{I}}$ is $f$-maximizing.
		
		Since the space $\mathcal{M}(X)$ of Borel probability measures on $X$ is weak*-compact, it follows that there exists a weak*-convergent sub-net $\left( \mu_{i_j} \right)_{j \in \mathscr{J}}$, converging to some $\mu$. It follows then that $\mu$ is $T$-invariant, since if $f_0 \in C(X) , g_0 \in G$, then
		\begin{align*}
			\left| \int T_{g_0} f_0 \mathrm{d} \mu - \int f_0 \mathrm{d} \mu \right|	& = \left| \int \left( T_{g_0} f_0 - f_0 \right) \mathrm{d} \mu \right| \\
			& = \lim_j \left| \int \left( T_{g_0} f_0 - f_0 \right) \mathrm{d} \mu_{i_j} \right| ,
		\end{align*}
		where
		\begin{align*}
			& \left| \int_X \frac{1}{m\left(F_{i_j} \right)} \int_{F_{i_j}} \left( T_{g_0} f_0(x) - f_0(x) \right) \mathrm{d} m(g) \mathrm{d} \sigma_{i_j} (x)\right| \\
			=	& \left| \int_X \frac{1}{m \left( F_{i_j} \right)} \left( \left( \int_{ F_{i_j} g_0 \setminus F_{i_j}} T_g f_0(x) \mathrm{d} m(g) \right) - \left( \int_{F_{i_j} \setminus F_{i_j} g_0} T_g f_0(x) \mathrm{d} m(g) \right) \right) \mathrm{d} \sigma_{i_j} (x) \right| \\
			\leq	& \left| \int_X \frac{1}{m\left(F_{i_j}\right)} \int_{g_0 F_{i_j} \setminus F_{i_j}} T_g f_0(x) \mathrm{d} m(g) \mathrm{d} \sigma_{i_j}(x) \right| \\
			& + \left| \int_X \frac{1}{m\left(F_{i_j}\right)} \int_{F_{i_j} \setminus g_0 F_{i_j}} T_g f_0(x) \mathrm{d} m(g) \mathrm{d} \sigma_{i_j}(x) \right| \\
			\leq & \frac{m \left(g_0 F_{i_j} \setminus F_{i_j} \right) + m \left( F_{i_j} \setminus g_0 F_{i_j} \right)}{m \left( F_{i_j} \right)} \left\| f_0 \right\|_{C(X)} \\
			= & \frac{ m \left( F_{i_j} g_0 \Delta F_{i_j} \right) }{m \left( F_{i_j}\right)} \| f_0 \|_{C(X)} \\
			\stackrel{j \to \infty}{\to}	& 0 .
		\end{align*}
		Therefore $\int T_{g_0} f_0 \mathrm{d} \mu = \int f_0 \mathrm{d} \mu$, meaning $\mu$ is $T$-invariant.
		
		We claim that $\mu$ is $f$-maximizing. On one hand, we know that $\int f \mathrm{d} \mu \leq \Gamma(f)$, because $\mu \in \mathcal{M}_T(X)$. Now, suppose that $\nu \in \mathcal{M}_T(X)$. Then
		\begin{align*}
			\int f \mathrm{d} \nu	& = \int_X \frac{1}{m \left( F_{i_j} \right)} T_g f(x) \mathrm{d} m(g) \mathrm{d} \nu(x) \\
			& \leq \left\| \frac{1}{m \left( F_{i_j} \right)} T_g f \mathrm{d} m(g) \right\|_{C(X)} \\
			& = \int f \mathrm{d} \mu_{i_j} \\
			\Rightarrow \int f \mathrm{d} \nu	& \leq \lim_j \int f \mathrm{d} \mu_{i_j} \\
			& = \int f \mathrm{d} \mu .
		\end{align*}
		Therefore $\int f \mathrm{d} \mu \geq \int f \mathrm{d} \nu$ for all $\nu \in \mathcal{M}_T(X)$, meaning that $\int f \mathrm{d} \mu = \sup_{\nu \in \mathcal{M}_T(X)} \int f \mathrm{d} \nu$, i.e. $\mu$ is $f$-maximizing, so $\int f \mathrm{d} \mu = \Gamma(f)$.
	\end{proof}
	
	From this, we can use the gauge to provide a characterization of uniquely ergodic systems.
	
	\begin{Thm}\label{Unique ergodicity and gauge}
		Let $T : G \curvearrowright X$ be a topological dynamical system, where $G$ is amenable, and let $\mu \in \mathcal{M}_T(X)$ be a $T$-invariant Borel probability measure on $X$ that is fully supported on $X$, i.e. gives positive measure to every nonempty open subset of $X$. Then $T : G \curvearrowright X$ is uniquely ergodic if and only if $\Gamma(f) = \int f \mathrm{d} \mu$ for all nonnegative-valued $f \in C_{\mathbb{R}}(X)$.
	\end{Thm}
	
	\begin{proof}
		Clearly $\int f \mathrm{d} \mu \leq \Gamma(f)$ for all $f \in C_\mathbb{R}(X)$.
		
		$(\Rightarrow)$ If $T : G \curvearrowright X$ is uniquely ergodic, then $\mu$ is $f$-maximizing for all $f \in C_\mathbb{R}(X)$, so in particular $\int f \mathrm{d} \mu = \Gamma(f)$ for all nonnegative $f \in C_\mathbb{R}(X)$.
		
		$(\Leftarrow)$ We'll prove the contrapositive. Suppose that $T : G \curvearrowright X$ is \emph{not} uniquely ergodic. Then there exists an ergodic $T$-invariant Borel probability measure $\nu \neq \mu$. By \cite[Theorem 1]{JenkinsonEvery}, there exists a continuous real-valued function $f \in C_\mathbb{R}(X)$ such that $\mathcal{M}_\mathrm{max}(f) = \{\nu\}$. By possibly adding a nonnegative constant to $f$, we can assume that $f$ is nonnegative. But $\Gamma(f) = \int f \mathrm{d} \nu \neq \int f \mathrm{d} \mu$.
	\end{proof}
	
	Finally, we want to provide a connection between unique ergodicity and temporo-spatial differentiation problems. Before stating the main theorem relating these, we prove the following lemma that relates the $\alpha_C$ functionals to the $L^\infty$ norm.
	
	\begin{Lem}\label{Quantitative estimates for convergence of STDs}
		Let $(X, \mu)$ be a probability space, and let $f \in L^\infty (X, \mu)$. Then
		$$\frac{1}{2} \|f\|_\infty \leq \sup \left\{ \left| \alpha_C \left( f \right) \right| : \textrm{$C \subseteq X$ measurable, } \mu(C) > 0 \right\} \leq \left\| f \right\|_\infty , $$
		and in particular, if $f$ is real-valued, then
		$$\sup \left\{ \left| \alpha_C (f) \right| : \textrm{$C \subseteq X$ measurable, } \mu(C) > 0 \right\} = \left\| f \right\|_\infty .$$
	\end{Lem}
	
	\begin{proof}
		In either case, it's clear that $|\alpha_C(f)| \leq \|f\|_\infty$ for all $C \subseteq X$ measurable with $\mu(C) > 0$, since $\alpha_C$ is a state on $L^\infty (X, \mu)$
		
		Consider now the case that $f$ is real-valued. If $\left\| f \right\|_\infty = 0$, then the equality is immediate, so suppose that $\left\| f \right\|_\infty > 0$. Set $f^+ = \max(f, 0), f^- = \max(-f, 0)$. Then $\|f\|_\infty = \max \left\{ \left\|f^+\right\|_\infty , \left\|f^-\right\|_\infty \right\}$. Assume without loss of generality that $\left\| f \right\|_\infty = \left\| f^+ \right\|_\infty$. For each $k \in \mathbb{N}$, set
		$$C_k = \left\{ x \in X : f^+(x) > \frac{k}{k + 1} \left\| f^+ \right\|_\infty \right\} = \left\{ x \in X : f(x) > \frac{k}{k + 1} \left\| f\right\|_\infty \right\} .$$
		Then $\mu(C_k) > 0$ and $\alpha_{C_k} \left( f \right) \geq \frac{k}{k + 1} \left\| f \right\|_\infty$ for all $k \in \mathbb{N}$, meaning in particular that
		$$\sup \left\{ \left| \alpha_C \left( f \right) \right| : \textrm{$C \subseteq X$ measurable, } \mu(C) > 0 \right\} \geq \left\| f \right\|_\infty .$$
		
		Now suppose that $f$ is not necessarily real-valued. Let $h_1, h_2 \in L_\mathbb{R}^\infty(X, \mu)$ be the real and imaginary parts of $f$, respectively. Then $\|f\|_\infty \leq \|h_1\|_\infty + \|h_2\|_\infty$. Therefore $\max \left\{ \|h_1\|_\infty , \|h_2\|_\infty \right\} \geq \frac{1}{2} \|f\|_\infty$. Assume without loss of generality that $\|h_1\|_\infty \geq \|h_2\|_\infty$, so $\|h_1\|_\infty \geq \frac{1}{2} \|f\|_\infty$. For each $k \in \mathbb{N}$, choose $C_k' \subseteq X$ measurable such that $\mu \left( C_k' \right) > 0$, and $\alpha_{C_k'}(h_1) \geq \frac{k}{k + 1} \|h_1\|_\infty$, which is possible if we appeal to the real case. Then
		\begin{align*}
			\left| \alpha_{C_k'}(f) \right|	& = \left| \alpha_{C_k'}(h_1) + i \alpha_{C_k'}(h_2) \right| \\
			& \geq \left| \alpha_{C_k'}(h_1) \right| \\
			& \geq \frac{k}{k + 1} \|h_1\|_\infty \\
			& \geq \frac{k}{k + 1} \left( \frac{1}{2} \|f\|_\infty \right) .
		\end{align*}
		Taking the limit as $k \to \infty$ verifies that
		\begin{align*}
			\sup \left\{ \left| \alpha_C \left( f \right) \right| : \textrm{$C \subseteq X$ measurable, } \mu(C) > 0 \right\}	& \geq \frac{1}{2} \|f\|_\infty .
		\end{align*}
	\end{proof}
	
	In the case that $X$ is a compact metrizable space, the measure $\mu$ is Borel, and the $f$ is continuous, Lemma \ref{Quantitative estimates for convergence of STDs} can be sharpened as follows.
	
	\begin{Lem}
		Let $(X, \mu)$ be a probability space, where $X$ is a compact metrizable space and $\mu$ is a Borel probability measure. Let $f \in C (X)$. Then
		$$\frac{1}{2} \|f\|_{\infty} \leq \sup \left\{ \left| \alpha_C \left( f \right) \right| : \textrm{$C \subseteq X$ open, } \mu(C) > 0 \right\} \leq \left\| f \right\|_\infty , $$
		and in particular, if $f$ is real-valued, then we have
		$$\sup \left\{ \left| \alpha_C (f) \right| : \textrm{$C \subseteq X$ open, } \mu(C) > 0 \right\} = \left\| f \right\|_\infty .$$
	\end{Lem}
	
	\begin{proof}
		Under these conditions, all the $C_k$ and $C_k'$ in the proof of Lemma \ref{Quantitative estimates for convergence of STDs} are open. The result follows from the same proof.
	\end{proof}
	
	A natural corollary of Lemma \ref{Quantitative estimates for convergence of STDs} is the following qualitative statement.
	
	\begin{Thm}\label{Description of uniformity}
		Let $T : G \curvearrowright (X, \mu)$ be a measure-preserving action of a discrete (not necessarily amenable) group $G$ on a probability space $(X, \mu)$. Let $(F_i)_{i \in \mathscr{I}}$ be a net of compact subsets of $G$ with positive measure. Let $f \in L^\infty(X, \mu)$. Then the following conditions are equivalent.
		\begin{enumerate}[label=(\roman*)]
			\item $\operatorname{Avg}_{F_i} f \to \int f \mathrm{d} \mu$ in the norm topology on $L^\infty$.
			\item $\alpha_{C_i} \left( \operatorname{Avg}_{F_i} f \right) \to \int f \mathrm{d} \mu$ for all nets of measurable subsets $C_i$ of $X$ with positive measure.
		\end{enumerate}
	\end{Thm}
	
	\begin{proof}
		This equivalence follows from the estimates in Lemma \ref{Quantitative estimates for convergence of STDs}. For each $i \in \mathscr{I}$, set
		$$f_i = \operatorname{Avg}_{F_i} f - \int f \mathrm{d} \mu .$$
		
		(i)$\Rightarrow$(ii) Suppose that $f_i \to \int f \mathrm{d} \mu$ in $L^\infty$, and let $(C_i)_{i \in \mathscr{I}}$ be a net of measurable subsets $C_i$ of $X$ with positive measure. Then
		$$
		\left| \alpha_{C_i} \left( \operatorname{Avg}_{F_i} f \right) - \int f \mathrm{d} \mu \right| = \left| \alpha_{C_i}\left( f_i \right) \right| \leq \left\| f_i \right\|_\infty \stackrel{i \to \infty}{\to} 0 . $$
		
		(ii)$\Rightarrow$(i) We'll prove $\neg$(i)$\Rightarrow \neg$(ii). Suppose that $\limsup_i \left\| f_i \right\|_\infty > 0$. For each $i \in \mathscr{I}$, choose $C_i \subseteq X$ measurable with positive measure such that 
		$$\left| \alpha_{C_i} \left( f_i \right) \right| \geq \frac{1}{2} \sup \left\{ \left| \alpha_C \left( f_i \right) \right| : \textrm{$C \subseteq X$ measurable, } \mu(C) > 0 \right\} \geq \frac{1}{4} \left\| f_i\right\|_\infty .$$
		Then $\limsup_i \left| \alpha_{C_i} \left( \operatorname{Avg}_{F_i} f \right) - \int f \mathrm{d} \mu \right| = \limsup_i \left| \alpha_{C_i} \left( f_i \right) \right| \geq \frac{1}{4} \limsup_i \left\| f_i \right\|_\infty > 0$.
	\end{proof}
	
	Theorem \ref{Description of uniformity} gives a qualitative description of the conditions under which we get the ``best possible" behavior for a temporo-spatial differentiation problem, i.e. conditions under which $\alpha_{C_i} \left( \operatorname{Avg}_{F_i} f \right) \to \int f \mathrm{d} \mu$ independent of the choice of $C_i$. However, Lemma \ref{Quantitative estimates for convergence of STDs} provides a potential avenue for quantitative estimates on the rate of convergence for temporo-spatial averages by ``importing" estimates on the rate of $L^\infty$-convergence for $\operatorname{Avg}_{F_i} \to \int f \mathrm{d} \mu$.
	
	In general, the classical ergodic theorems don't give us estimates on the rate of convergence they promise, and this convergence can in fact be very slow. See the ``Speed of Convergence" discussion in $\S 1.2$ of \cite{Krengel} for a survey of relevant counterexamples. However, some authors have studied situations where effective estimates on the convergence of certain ergodic averages can be obtained. For a rudimentary example of this type, consider the case where $\alpha \in \mathbb{R} \setminus \mathbb{Q}$ is an irrational real, and $T : \mathbb{Z} \curvearrowright \left( \mathbb{R} / \mathbb{Z} \right)$ is an action of $\mathbb{Z}$ on the circle by $T x = \alpha + x$, where $\mathbb{R} / \mathbb{Z}$ is endowed with its Haar probability measure $\mu$. Then by appealing to the unique ergodicity of $(X, T)$, we can say that $\frac{1}{k} \sum_{j = 0}^{k - 1} T^j f \to \int f \mathrm{d} \mu$ in $L^\infty(X, \mu)$ for all $f \in C(X)$. However, if $f(x) = e^{2 \pi i n x}$ for some $n \in \mathbb{Z} \setminus \{0\}$, i.e. if $f$ is a nontrivial character on $\mathbb{R} / \mathbb{Z}$, then using a geometric series, we can see that $\left\| \frac{1}{k} \sum_{j = 0}^{k - 1} T^j f - \int f \mathrm{d} \mu \right\|_\infty \leq A_{\alpha, n} k^{-1}$ for some constant $A_{\alpha, n} \in (0, \infty)$, yielding a quantitative estimate on that convergence rate. In particular, Lemma \ref{Quantitative estimates for convergence of STDs} tells us that under those circumstances, we'd have that
	$$\left| \alpha_{C_k} \left( \frac{1}{k} \sum_{j = 0}^{k - 1} T^j f \right) \right| \leq k^{-1} A_{\alpha, n} $$
	for all choices of $(C_k)_{k = 1}^\infty$. In this article, we will say no more on this topic, which is linked to the study of effective equidistribution (see \cite{EffectiveEquidistribution}).
	
	\begin{Lem}\label{Herman ergodic theorem in amenable setting}
		Let $T : G \curvearrowright X$ be a topological dynamical system, where $G$ is amenable. Let $(F_i)_{i \in \mathscr{I}}$ be a \Folner \space net for $G$, and let $f \in C(X), \lambda \in \mathbb{C}$. Then the following conditions are related by the implications (i)$\iff$(ii)$\Rightarrow$(iii). If in addition we have that $\mathscr{I} = \mathbb{N}$, i.e. that $(F_i)_{i \in \mathbb{N}}$ is a \Folner \space \emph{sequence}, then (iii)$\Rightarrow$(i).
		\begin{enumerate}[label=(\roman*)]
			\item $\int f \mathrm{d} \mu = \lambda$ for all $T$-invariant Borel probability measures $\mu$ on $X$.
			\item $\operatorname{Avg}_{F_i} f \to \lambda$ uniformly.
			\item $\operatorname{Avg}_{F_i} f(x) \to \lambda$ for all $x \in X$.
		\end{enumerate}
	\end{Lem}
	
	\begin{proof}
		(i)$\Rightarrow$(ii): Suppose that $\int f \mathrm{d} \mu = \lambda$ for all $T$-invariant Borel probabiliy measures $\mu$ on $X$, and let $(x_i)_{i \in \mathscr{I}}$ be a net in $X$ such that
		\begin{align*}
			\left| \operatorname{Avg}_{F_i} f(x_i) - \lambda \right|	& = \left\| \operatorname{Avg}_{F_i} f - \lambda \right\|_{C(X)}	& (\forall i \in \mathscr{I}) .
		\end{align*}
		Let $(\mu_i)_{i \in I}$ be the net of Borel probability measures on $X$ given by
		\begin{align*}
			\int g \mathrm{d} \mu_i	& = \operatorname{Avg}_{F_i} g(x_i)	& (\forall g \in C(X), i \in \mathscr{I}) .
		\end{align*}
		Appealing to compactness, let $\left( \mu_{i_j} \right)_{j \in \mathscr{J}}$ be a weak*-convergent subnet along which
		$$\lim_j \left| \int f \mathrm{d} \mu_{i_j} - \lambda \right| = \limsup_i \left| \int f \mathrm{d} \mu_i - \lambda \right| .$$
		Let $\mu = \lim_j \mu_{i_j}$. Since $(F_{i_j})_{j \in \mathscr{J}}$ is \Folner, it follows from a classical argument that $\mu$ is $T$-invariant, and
		$$\left| \int f \mathrm{d} \mu - \lambda \right| = \limsup_i \left| \int f \mathrm{d} \mu_i - \lambda \right| = \limsup_i \left\| \operatorname{Avg}_{F_i} f - \lambda \right\|_{C(X)} .$$
		But $\int f \mathrm{d} \mu = \lambda$ by (i), so it follows that $\limsup_i \left\| \operatorname{Avg}_{F_i} f - \lambda \right\|_{C(X)} = 0$, meaning that $\operatorname{Avg}_{F_i} f \to \lambda$ uniformly.
		
		(ii)$\Rightarrow$(i): Trivial.
		
		(ii)$\Rightarrow$(iii): Trivial.
		
		(iii)$\Rightarrow$(i): $(F_i)_{i \in \mathbb{N}}$ is a \Folner \space \emph{sequence}, and let $\mu$ be a Borel probability measure on $X$. Then $\operatorname{Avg}_{F_i} f \stackrel{i \to \infty}{\to} \lambda$ pointwise-almost everywhere, and the functions $\operatorname{Avg}_{F_i} f$ are dominated by the constant function $\|f\|_{C(X)}$, so we can appeal to the Dominated Convergence Theorem to say that
		$$\int f \mathrm{d} \mu = \int \operatorname{Avg}_{F_i} f \mathrm{d} \mu \stackrel{i \to \infty}{\to} \int \lambda \mathrm{d} \mu = \lambda .$$
	\end{proof}
	
	\begin{Rmk}
		The reason we add the caveat that $\mathscr{I} = \mathbb{N}$ to ensure that (iii)$\Rightarrow$(i) in our proof of Lemma \ref{Herman ergodic theorem in amenable setting} is that there is in general no Dominated Convergence Theorem for arbitrary nets. For an elementary example, let $\mathscr{I} = \mathcal{P}_F([0, 1])$ be the net of finite subsets of $[0, 1]$, and define for each $i \in \mathscr{I}$, and for each $S \in \mathscr{I}$, let $f_i \in C(X)$ be a continuous function such that $f_i \vert_i \equiv 1$ and $\int f_i \mathrm{d} \mu \leq 1 / 2$, where $\mu$ is the Lebesgue probability measure on $[0, 1]$. Then $\lim_i f_i(x) = 1$ for all $x \in [0, 1]$, but $\limsup_i \int f_i \mathrm{d} \mu \leq 1/2$.
	\end{Rmk}
	
	The equivalence (i)$\iff$(ii) of Lemma \ref{Herman ergodic theorem in amenable setting} in the case where $G = \mathbb{Z}, \; F_k = \{0, 1, \ldots, k - 1\}$ can be found in \cite[Lemme on pg. 487]{Herman}. This result generalizes the classical result of Oxtoby \cite[(5.3)]{OxtobyErgodic} relating unique ergodicity and uniform convergence of temporal averages, as unique ergodicity is equivalent to $\left\{ \int f \mathrm{d} \nu : \nu \in \mathcal{M}_T(X) \right\}$ being singleton for all $f \in C(X)$. Since this property will be important for the remainder of this section, we introduce the following definition.
	
	\begin{Def}
		Let $T : G \curvearrowright X$ be a topological dynamical system, and let $f \in C(X)$. We say that $f$ is \emph{$T$-Herman} (or simply \emph{Herman}, when $T$ is clear from context) if $\left\{ \int f \mathrm{d} \nu : \nu \in \mathcal{M}_T(X) \right\}$ is singleton.
	\end{Def}
	
	The following theorem tells us that the best kind of convergence for temporo-spatial differentiations can be characterized in terms of ergodic optimization.
	
	\begin{Thm}\label{Herman-type result}
		Let $T : G \curvearrowright X$ be a topological dynamical system, where $G$ is amenable. Let $\mu \in \mathcal{M}_T(X)$ be a $T$-invariant Borel probability measure on $X$. Let $(F_i)_{i \in \mathscr{I}}$ be a \Folner \space net for $G$, and let $f \in C(X)$. Then the following conditions are related by the implications (1)$\Rightarrow $(2)$\Rightarrow$(3), and if $\mu$ is fully supported on $X$, then (3)$\Rightarrow$(1).
		\begin{enumerate}
			\item $f$ is Herman.
			\item For every net $(C_i)_{i \in \mathscr{I}}$ of Borel-measurable sets $C_i$ of positive measure, the net
			$$\left( \alpha_{C_i} \left( \operatorname{Avg}_{F_i} f \right) \right)_{i \in \mathscr{I}}$$
			converges to $\int f \mathrm{d} \mu$.
			\item For every net $(U_i)_{i \in \mathscr{I}}$ of open sets $U_i$ of positive measure, the net
			$$\left( \alpha_{U_i} \left( \operatorname{Avg}_{F_i} f \right) \right)_{i \in \mathscr{I}}$$
			converges to $\int f \mathrm{d} \mu$.
		\end{enumerate}
	\end{Thm}
	
	\begin{proof}
		(1)$\Rightarrow$(2) Suppose that $f$ is Herman, and let $(F_i)_{i \in \mathscr{I}}$ be a \Folner \space net for $G$. Then by Lemma \ref{Herman ergodic theorem in amenable setting}, the net $\left( \operatorname{Avg}_{F_i} f \right)_{i \in \mathscr{I}}$ converges in $C(X)$-norm to $\int f \mathrm{d} \mu$, and since $\| \cdot \|_\infty \leq \|\cdot\|_{C(X)}$, it follows that $\operatorname{Avg}_{F_i} f \to \int f \mathrm{d} \mu$ in $L^\infty(X, \mu)$. Therefore (1)$\Rightarrow$(2) follows from Theorem \ref{Description of uniformity}.
		
		(2)$\Rightarrow$(3) Trivial.
		
		$\neg$(1)$\Rightarrow \neg$(3) Suppose that $\mu$ is fully supported. For this direction, we can assume that $f$ is real-valued, since otherwise we can break $f$ into its real and imaginary parts and consider those parts separately. So for the remainder of this proof, we can assume that $f$ is real-valued.
		
		Suppose that $f$ is not Herman, and that $\mu$ is strictly positive. Set
		\begin{align*}
			m_1	& = \min \left\{ \int f \mathrm{d} \nu : \nu \in \mathcal{M}_T(X) \right\} , \\
			m_2	& = \max \left\{ \int f \mathrm{d} \nu : \nu \in \mathcal{M}_T(X) \right\} .
		\end{align*}
		If $\left\{ \int f \mathrm{d} \mu : \mathcal{M}_T(X) \right\}$ is not singleton, then $m_1 < m_2$, and in particular this tells us that at least one of the inequalities $\int f \mathrm{d} \mu < m_2, \int f \mathrm{d} \mu > m_1$ is true. We consider two cases:

		\textbf{Case (i):} Consider the case where $m_2 > \int f \mathrm{d} \mu$. Set $g = f + \|f\|_{C(X)}$, which is a nonnegative-valued function with
		$$\Gamma(g) = m_2 + \|f\|_{C(X)} > \int f \mathrm{d} \mu + \|f\|_{C(X)} = \int g \mathrm{d} \mu .$$
		Choose $L \in \left( \int g \mathrm{d} \mu , \Gamma(g) \right)$. For each $i \in \mathscr{I}$, set
		$$U_i = \begin{cases}
			\left\{ x \in X : \operatorname{Avg}_{F_i} g(x) > L \right\}	& \textrm{if $\left\| \operatorname{Avg}_{F_i} g \right\|_{C(X)} > L$} , \\
			X	& \textrm{if $\left\| \operatorname{Avg}_{F_i} g \right\|_{C(X)} \leq L$} .
		\end{cases}$$
		Because each $U_i$ is a nonempty open set, and $\mu$ is fully supported, we know that each $U_i$ has positive measure. By Theorem \ref{Gauge converges}, we know that
		$$\lim_i \left\| \operatorname{Avg}_{F_i} g \right\| = \Gamma(g) > L > \int g \mathrm{d} \mu.$$ Thus $\limsup_i \alpha_{U_i} \left( \operatorname{Avg}_{F_i} g \right) \geq L > \int g \mathrm{d} \mu = \int f \mathrm{d} \mu + \|f\|_{C(X)}$, so
		$$\limsup_i \alpha_{U_i} \left( \operatorname{Avg}_{F_i} f \right) = \limsup_i \alpha_{U_i} \left( \operatorname{Avg}_{F_i} g - \|f\|_{C(X)} \right) > \int f \mathrm{d} \mu . $$
		
		\textbf{Case (ii):} Suppose $m_1 < \int f \mathrm{d} \mu$. Consider $g = \|f\|_{C(X)} - f$, a nonnegative-valued function. Then for $\nu \in \mathcal{M}_T(X)$, we have
		\begin{align*}
			\int g \mathrm{d} \nu	& = \|f\|_{C(X)} - \int f \mathrm{d} \nu \\
			\Rightarrow \Gamma(g)	& = \|f\|_{C(X)} - m_1 \\
			& > \|f\|_{C(X)} - \int f \mathrm{d} \mu \\
			& = \int g \mathrm{d} \mu .
		\end{align*}
		Choose $L \in (\int g \mathrm{d} \mu, \Gamma(g))$. Construct open subsets $U_i$ of $X$ by
		$$U_i = \begin{cases}
			\left\{ x \in X : \operatorname{Avg}_{F_i} g(x) > L \right\}	& \textrm{if $\left\| \operatorname{Avg}_{F_i} g \right\| > L$} , \\
			X	& \textrm{if $\left\| \operatorname{Avg}_{F_i} g \right\| \leq L$} .
		\end{cases}$$
		Then by a similar argument to that used in Case (i), we know that \linebreak$\limsup_i \alpha_{U_i} \left( \operatorname{Avg}_{F_i} g \right) \geq L > \int g \mathrm{d} \mu$. It then follows that
		\begin{align*}
			\liminf_i \alpha_{U_i} \left( \operatorname{Avg}_{F_i} f \right)	& = \liminf_i \alpha_{U_i} \left( \operatorname{Avg}_{F_i} \left( \|f\|_{C(X)} - g \right) \right) \\
			& = \|f\|_{C(X)} - \limsup_i \alpha_{U_i} (\operatorname{Avg}_{F_i} g) \\
			& < \|f\|_{C(X)} - \int g \mathrm{d} \mu \\
			& = \int f \mathrm{d} \mu .
		\end{align*}
	\end{proof}
	
	We now come to a theorem which provides a qualitative connection between unique ergodicity and temporo-spatial differentiation problems.
	
	\begin{Thm}\label{(1) implies (2) implies (3)}
		Let $T : G \curvearrowright X$ be a topological dynamical system, where $G$ is amenable. Let $\mu \in \mathcal{M}_T(X)$ be a $T$-invariant Borel probability on $X$. Let $(F_i)_{i \in \mathscr{I}}$ be a \Folner \space net for $G$. Then the following conditions are related by the implications (1)$\Rightarrow $(2)$\Rightarrow$(3), and if $\mu$ is fully supported on $X$, then (3)$\Rightarrow$(1).
		\begin{enumerate}
			\item $T : G \curvearrowright X$ is uniquely ergodic.
			\item For every net $(C_i)_{i \in \mathscr{I}}$ of Borel-measurable sets $C_i$ of positive measure, the net
			$$\left( \alpha_{C_i} \left( \operatorname{Avg}_{F_i} f \right) \right)_{i \in \mathscr{I}}$$
			converges to $\int f \mathrm{d} \mu$ for all $f \in C(X)$.
			\item For every net $(U_i)_{i \in \mathscr{I}}$ of open sets $U_i$ of positive measure, the net
			$$\left( \alpha_{U_i} \left( \operatorname{Avg}_{F_i} f \right) \right)_{i \in \mathscr{I}}$$
			converges to $\int f \mathrm{d} \mu$ for all $f \in C(X)$.
		\end{enumerate}
	\end{Thm}
	
	\begin{proof}
		(1)$\Rightarrow$(2) The unique ergodicity of $T : G \curvearrowright X$ is equivalent to every $f \in C(X)$ being Herman. Apply Theorem \ref{Herman-type result}.
		
		(2)$\Rightarrow$(3) Trivial.
		
		$\neg$(1)$\Rightarrow \neg$(3) Suppose that $T : G \curvearrowright X$ is not uniquely ergodic, and that $\mu$ is strictly positive. By Theorem \ref{Unique ergodicity and gauge}, there exists a nonnegative-valued $f \in C_\mathbb{R} (X)$ such that $\int f \mathrm{d} \mu < \Gamma(f) = \lim_i \left\| \operatorname{Avg}_{F_i} f \right\|$. Let $L \in \left( \int f \mathrm{d} \mu , \Gamma(f) \right)$. For each $i \in \mathscr{I}$, set
		$$U_i = \begin{cases}
			\left\{ x \in X : \operatorname{Avg}_{F_i} f(x) > L \right\}	& \textrm{if $\left\| \operatorname{Avg}_{F_i} f \right\| > L$} , \\
			X	& \textrm{if $\left\| \operatorname{Avg}_{F_i} f \right\| \leq L$} .
		\end{cases}$$
		Because each $U_i$ is a nonempty open set, and $\mu$ is fully supported, we know that each $U_i$ has positive measure. Thus $\limsup_i \alpha_{U_i} \left( \operatorname{Avg}_{F_i} f \right) \geq L > \int f \mathrm{d} \mu$.
	\end{proof}
	
	In the event that we're dealing not just with a \Folner \space net, but instead a \Folner \space \emph{sequence}, we can make a stronger claim: that unique ergodicity is equivalent to all the temporo-spatial differentiations of continuous functions by that temporal averaging sequence converging.
	
	\begin{Thm}\label{Herman extension of the convergence result}
		Let $T : G \curvearrowright X$ be a topological dynamical system, where $G$ is amenable. Let $\mu \in \mathcal{M}_T(X)$ be a $T$-invariant Borel probability measure on $X$, and let $(F_k)_{k = 1}^\infty$ be a \Folner \space sequence. Let $f \in C(X)$. Then the following conditions are related by the implications (1)$\Rightarrow $(2)$\Rightarrow$(3)$\Rightarrow$(4), and if $\mu$ is fully supported on $X$, then (4)$\Rightarrow$(1).
		\begin{enumerate}
			\item $f$ is Herman.
			\item For every sequence $(C_k)_{k = 1}^\infty$ of Borel-measurable sets $C_k$ of positive measure, the sequence
			$$\left( \alpha_{C_k} \left( \operatorname{Avg}_{F_k} f \right) \right)_{k = 1}^\infty$$
			converges to $\int f \mathrm{d} \mu$.
			\item For every sequence $(U_k)_{k = 1}^\infty$ of open sets $U_k$ of positive measure, the sequence
			$$\left( \alpha_{U_k} \left( \operatorname{Avg}_{F_k} f \right) \right)_{k = 1}^\infty$$
			converges to $\int f \mathrm{d} \mu$.
			\item For every sequence $(U_k)_{k = 1}^\infty$ of open sets $U_k$ of positive measure, the sequence
			$$\left( \alpha_{U_k} \left( \operatorname{Avg}_{F_k} f \right) \right)_{k = 1}^\infty$$
			converges to some complex number.
		\end{enumerate}
		
		Furthermore, if $\left\{ \int f \mathrm{d} \nu : \nu \in \mathcal{M}_T(X) \right\}$ is not singleton, the measure $\mu$ is fully supported, and the space $(X, \mu)$ is atomless, then we can choose a sequence $\left(U_k'\right)_{k = 1}^\infty$ of open subsets of $X$ with positive measure and a continuous $f \in C(X)$ such that $\left( \alpha_{U_k'} \left( \operatorname{Avg}_{F_k} f \right) \right)_{k = 1}^\infty$ diverges and $\mu\left( U_k' \right) \searrow 0$.
	\end{Thm}
	
	\begin{proof}
		That (1)$\Rightarrow$(2)$\Rightarrow$(3) follows immediately from Theorem \ref{Herman-type result}, and (3)$\Rightarrow$(4) is trivial. Now we'll show that if $\mu$ is fully supported, then $\neg$(1)$\Rightarrow \neg$(4). Suppose that $f$ is \emph{not} Herman. We can consider the case where $f$ is real-valued, since otherwise we can break $f$ into its real and imaginary parts. Moreover, we can assume that $f$ is nonnegative-valued, since otherwise we can just replace $f$ with $f + \|f\|_{C(X)}$. So for the remainder of this proof, we assume that $f$ is nonnegative-valued.
		
		Set
		\begin{align*}
			m_1	& = \min \left\{ \int f \mathrm{d} \nu : \nu \in \mathcal{M}_T(X) \right\} , \\
			m_2	& = \max \left\{ \int f \mathrm{d} \nu : \nu \in \mathcal{M}_T(X) \right\} .
		\end{align*}
		If $\left\{ \int f \mathrm{d} \mu : \mathcal{M}_T(X) \right\}$ is not singleton, then $m_1 < m_2$, and in particular this tells us that at least one of the inequalities $\int f \mathrm{d} \mu < m_2, \int f \mathrm{d} \mu > m_1$ is true.

		\textbf{Case (i):} Consider first the case where $m_2 > \int f \mathrm{d} \mu$. Choose $L, M \in \mathbb{R}$ such that $\int f \mathrm{d} \mu < L < M < \Gamma(f)$. Define open sets $V_k, W_k \subseteq X$ for $k \in \mathbb{N}$ by
		\begin{align*}
			V_k	& = \left\{ x \in X : \operatorname{Avg}_{F_k} f(x) > M \right\} , \\
			W_k	& = \left\{ x \in X : \operatorname{Avg}_{F_k} f(x) < L \right\} .
		\end{align*}
		Both sets are obviously open, since they're preimages of open subsets of $\mathbb{R}$ under the continuous functions $\operatorname{Avg}_{F_k} f \in C_\mathbb{R}(X)$.
		
		First, we know that there exists $K \in \mathbb{N}$ such that $V_k \neq \emptyset$ for all $k \geq K$. This is because we know there exists $K \in \mathbb{N}$ in $X$ such that $\left\| \operatorname{Avg}_{F_k} f \right\|_{C(X)} > M$ for all $k \geq K$, and by the Extreme Value Theorem, we know there exists a sequence $(x_k)_{k = 1}^\infty$ such that
		$$\left\| \operatorname{Avg}_{F_k} f(x) \right\|_{C(X)} = \operatorname{Avg}_{F_k} f(x_k)$$
		for all $k \in \mathbb{N}$. In particular, if $k \geq K$, then $x_k \in V_k$. Therefore $V_k \neq \emptyset$ for all $k \geq K$.
		
		Secondly, we claim that $W_k$ is nonempty for all $k \in \mathbb{N}$. To see this, suppose to the contrary that $W_k = \emptyset$ for some $k \in \mathbb{N}$. Then $f(x) \geq L > \int f \mathrm{d} \mu$ for all $x \in X$, meaning that $\int f \mathrm{d} \mu \geq L > \int f \mathrm{d} \mu$, a clear contradiction. So $W_k \neq \emptyset$ for all $k \in \mathbb{N}$.
		
		Now, define a sequence $(U_k)_{k = 1}^\infty$ of nonempty open subsets of $X$ by
		$$
		U_k = \begin{cases}
			X	& \textrm{if } k < K , \\
			V_k	& \textrm{if } k \geq K \textrm{ is odd}, \\
			W_k	& \textrm{if } k \geq K	\textrm{ is even} .
		\end{cases}
		$$
		Then
		\begin{align*}
			\limsup_{k \to \infty} \alpha_{U_k} \left( \operatorname{Avg}_{F_k} f \right)	& \geq \limsup_{k \to \infty} \alpha_{U_{2k + 1}} \left( \operatorname{Avg}_{F_{2k + 1}} f \right) \\
			& = \limsup_{k \to \infty} \alpha_{V_{2k + 1}} \left( \operatorname{Avg}_{F_{2k + 1}} f \right) \\
			& \geq \limsup_{k \to \infty} \alpha_{V_{2k + 1}} \left( M \right)\\
			& = M , \\
			\liminf_{k \to \infty} \alpha_{U_k} \left( \operatorname{Avg}_{F_k} f \right)	& \leq \liminf_{k \to \infty} \alpha_{U_{2k}} \left( \operatorname{Avg}_{F_{2k}} f \right) \\
			& = \liminf_{k \to \infty} \alpha_{W_{2k}} \left( \operatorname{Avg}_{F_{2k}} f \right) \\
			& \leq \liminf_{k \to \infty} \alpha_{W_{2k}} \left( L \right) \\
			& = L .
		\end{align*}
		Therefore
		$$\liminf_{k \to \infty} \alpha_{U_k} \left( \operatorname{Avg}_{F_k} f \right) \leq L < M \leq \limsup_{k \to \infty} \alpha_{U_k} \left( \operatorname{Avg}_{F_k} f \right) ,$$
		meaning the sequence diverges.
		
		\textbf{Case (ii):} Consider now the case where $m_1 < \int f \mathrm{d} \mu$. Replacing $f$ with $f' = \|f\|_{C(X)} - f$, another nonnegative-valued continuous function, we see that
		$$\int f' \mathrm{d} \mu = \|f\|_{C(X)} - \int f \mathrm{d} \mu < \|f\|_{C(X)} - m_1 = \max \left\{ \int f' \mathrm{d} \nu : \nu \in \mathcal{M}_T(X) \right\} .$$
		We can now carry out the construction from Case (i) on $f'$ instead of $f$ to get a sequence $(U_k)_{k = 1}^\infty$ of open sets along which $\left( \alpha_{U_k} \left( \operatorname{Avg}_{F_k} f' \right) \right)_{k = 1}^\infty$ diverges, and thus along which $\left( \alpha_{U_k} \left( \operatorname{Avg}_{F_k} f \right) \right)_{k = 1}^\infty$ diverges.
		
		Furthermore, if in addition, we assume that $(X, \mu)$ is atomless, then we can replace our $U_k$ with subsets $U_k'$ such that $\mu\left(U_k'\right) \searrow 0$. This can be done by recursively constructing a sequence of balls $U_k'$ contained in $U_k$ with sufficiently small radius that $\mu \left( U_{k + 1} ' \right) \leq \min \left\{ \mu \left( U_k ' \right) , 1/k \right\}$ for all $k \in \mathbb{N}$. This is possible by virtue of the atomlessness of $(X, \mu)$, since $0 = \mu(\{y_k\}) = \lim_{n \to \infty} \mu(B(y_k, 1/n))$. The above calculation will proceed the same way with the $U_k$ replaced by $U_k'$.
	\end{proof}
	
	\begin{Thm}\label{Weiss's extension}
		Let $T : G \curvearrowright X$ be a topological dynamical system, where $G$ is amenable. Let $\mu \in \mathcal{M}_T(X)$ be a $T$-invariant Borel probability on $X$, and let $(F_k)_{k = 1}^\infty$ be a \Folner \space sequence. Then the following conditions are related by the implications (1)$\Rightarrow $(2)$\Rightarrow$(3)$\Rightarrow$(4), and if $\mu$ is fully supported on $X$, then (4)$\Rightarrow$(1).
		\begin{enumerate}
			\item $T : G \curvearrowright X$ is uniquely ergodic.
			\item For every sequence $(C_k)_{k = 1}^\infty$ of Borel-measurable sets $C_k$ of positive measure, the sequence
			$$\left( \alpha_{C_k} \left( \operatorname{Avg}_{F_k} f \right) \right)_{k = 1}^\infty$$
			converges to $\int f \mathrm{d} \mu$ for all $f \in C(X)$.
			\item For every sequence $(U_k)_{k = 1}^\infty$ of open sets $U_k$ of positive measure, the sequence
			$$\left( \alpha_{U_k} \left( \operatorname{Avg}_{F_k} f \right) \right)_{k = 1}^\infty$$
			converges to $\int f \mathrm{d} \mu$ for all $f \in C(X)$.
			\item For every sequence $(U_k)_{k = 1}^\infty$ of open sets $U_k$ of positive measure, the sequence
			$$\left( \alpha_{U_k} \left( \operatorname{Avg}_{F_k} f \right) \right)_{k = 1}^\infty$$
			converges to some complex number for all $f \in C(X)$.
		\end{enumerate}
		
		Furthermore, if $(X, T)$ is not uniquely ergodic, the measure $\mu$ is fully supported, and the space $(X, \mu)$ is atomless, then we can choose a sequence $\left(U_k'\right)_{k = 1}^\infty$ of open subsets of $X$ with positive measure and a continuous $f \in C(X)$ such that $\left( \alpha_{U_k'} \left( \operatorname{Avg}_{F_k} f \right) \right)_{k = 1}^\infty$ diverges and $\mu\left( U_k' \right) \searrow 0$.
	\end{Thm}
	
	\begin{proof}
		That (1)$\Rightarrow$(2)$\Rightarrow$(3) follows immediately from Theorem \ref{(1) implies (2) implies (3)}, and (3)$\Rightarrow$(4) is trivial. Now we'll show that if $\mu$ is fully supported, then $\neg$(1)$\Rightarrow \neg$(4). Suppose that $T : G \curvearrowright X$ is \emph{not} uniquely ergodic. By Theorem \ref{Gauge converges}, there exists a nonnegative-valued $f \in C_\mathbb{R}(X)$ such that $\int f \mathrm{d} \mu < \Gamma(f) = \lim_{k \to \infty} \left\| \operatorname{Avg}_{F_k} f \right\|$, i.e. for which $\left\{ \int f \mathrm{d} \nu : \nu \in \mathcal{M}_T(X) \right\}$ is non-singleton. Appeal to Theorem \ref{Herman extension of the convergence result}.
	\end{proof}
	
	\begin{Rmk}
	Our Theorem \ref{Herman extension of the convergence result} generalizes Theorem 1.10 of \cite{Assani-Young}. We thank Benjamin Weiss for pointing out that connectedness was not necessary for that result.
	\end{Rmk}
	
	\section{Special cases of temporo-spatial differentiation problems}\label{Large averaging sets and static averaging sets}
	
	We digress here quickly to consider certain special classes of temporo-spatial differentiation problems: where the sequence of spatial averaging sets are constant, and where the spatial averaging sets have measure going to $1$.
	
	\begin{Prop}\label{Constant spatial averaging sequence}
		Let $T : G \curvearrowright (X, \mu)$ be a measure-preserving action of a discrete group $G$ on a probability space $(X, \mu)$. Let $f \in L^1(X, \mu)$, and let $(F_k)_{k = 1}^\infty$ be a sequence of nonempty finite subsets of $G$ such that the sequence $\left( \operatorname{Avg}_{F_k} f \right)$ converges to a function $f^* \in L^1(X, \mu)$ in the weak topology on $L^1(X, \mu)$. Then for every measurable subset $C$ of $X$ with positive measure, we have
		$$\alpha_C \left( \operatorname{Avg}_{F_k} f \right) \stackrel{k \to \infty}{\to} \alpha_C \left( f^* \right) .$$
	\end{Prop}
	
	\begin{proof}
		We know $\mu(C)^{-1} \chi_C \in L^\infty(X, \mu) = \left( L^1(X, \mu) \right) '$, so
		\begin{align*}
			\alpha_C \left( \operatorname{Avg}_{F_k} f \right)	& = \int \mu(C)^{-1} \chi_C \operatorname{Avg}_{F_k} f \mathrm{d} \mu	\\
			& = \left< \operatorname{Avg}_{F_k} f , \mu(C)^{-1} \chi_C \right> \\
			\left[ \operatorname{Avg}_{F_k} f \stackrel{k \to \infty}{\to} f^* \textrm{ in the weak topology on } L^1(X, \mu) \right]	& \stackrel{k \to \infty}{\to} \left< f^* , \mu(C)^{-1} \chi_C \right> \\
			& = \alpha_C \left( f^* \right) .
		\end{align*}
	\end{proof}
	
	We note that Proposition \ref{Constant spatial averaging sequence} is exceptional among all our temporo-spatial convergence results to date, in that it can be applied to a function $f$ which is not $L^\infty$, but merely $L^1$. It also brings us to the following corollary.
	
	\begin{Cor}
		Let $T : G \curvearrowright (X, \mu)$ be a measure-preserving action of a discrete amenable group $G$ on a probability space $(X, \mu)$. Let $f \in L^1(X, \mu)$, and let $(F_k)_{k = 1}^\infty$ be a \Folner \space sequence for $G$. Then for every measurable subset $C$ of $X$ of positive measure, we have
		$$\alpha_C \left( \operatorname{Avg}_{F_k} f \right) \stackrel{k \to \infty}{\to} \alpha_C \left( f^* \right) ,$$
		where $f^*$ is the projection of $f$ onto the subspace of invariant functions in $L^1 (X, \mu)$.
	\end{Cor}
	
	\begin{proof}
		This is a corollary of Proposition \ref{Constant spatial averaging sequence} and the Mean Ergodic Theorem for actions of amenable groups \cite[Theorem 4.23]{KerrLi}, since the norm topology on $L^1$ is stronger than the weak topology.
	\end{proof}
	
	\begin{Prop}\label{Measure to 1}
		Let $T : G \curvearrowright (X, \mu)$ be a measure-preserving action of a discrete group $G$ on a probability space $(X, \mu)$, and let $(C_k)_{k = 1}^\infty$ be a sequence of measurable subsets of $X$ such that $\mu(C_k) \to 1$. Let $(F_k)_{k = 1}^\infty$ be a sequence of nonempty finite subsets of $G$. Then
		$$\lim_{k \to \infty} \alpha_{C_k} \left( \operatorname{Avg}_{F_k} f \right) = \int f \mathrm{d} \mu$$
		for all $f \in L^\infty (X, \mu)$.
	\end{Prop}
	
	\begin{proof}
		Fix $f \in L^\infty(X, \mu)$. Then
		\begin{align*}
			& \left| \int f \mathrm{d} \mu - \alpha_{C_k} \left( \operatorname{Avg}_{F_k} f \right) \right|	\\
			=	& \left| \int_X \operatorname{Avg}_{F_k} \mathrm{d} \mu - \int_{C_k} \operatorname{Avg}_{F_k} f \mathrm{d} \mu + \int_{C_k} \operatorname{Avg}_{F_k} f \mathrm{d} \mu - \mu(C_k)^{-1} \int_{C_k} \operatorname{Avg}_{F_k} f \mathrm{d} \mu \right| \\
			\leq	& \left| \int_X \operatorname{Avg}_{F_k} f \mathrm{d} \mu - \int_{C_k} \operatorname{Avg}_{F_k} f \mathrm{d} \mu \right| + \left| \int_{C_k} \operatorname{Avg}_{F_k} f \mathrm{d} \mu - \mu(C_k)^{-1} \int_{C_k} \operatorname{Avg}_{F_k} f \mathrm{d} \mu \right| \\
			=	& \left| \int_{X \setminus C_k} \operatorname{Avg}_{F_k} f \mathrm{d} \mu \right| + \left( 1 - \mu(C_k)^{-1} \right) \left| \int_{C_k} \operatorname{Avg}_{F_k} f \mathrm{d} \mu \right| \\
			\leq	& \left( 1 - \mu(C_k) \right) \|f\|_\infty + \left( 1 - \mu(C_k)^{-1} \right) \mu(C_k) \|f\|_\infty \\
			\stackrel{k \to \infty}{\to}	& \infty .
		\end{align*}
	\end{proof}
	
	In light of Proposition \ref{Measure to 1}, we can see that temporo-spatial differentiation problems become trivial in the case where $\mu(C_k) \to 1$ for $f \in L^\infty(X, \mu)$. The result, however, fails for any unbounded integrable function. Let $f \in L^1(X, \mu) \setminus L^\infty(X, \mu)$, i.e. an unbounded integrable function, and let $E_k = \left\{ x \in X : |f(x)| \geq k \right\}$ for all $k \in \mathbb{N}$. Then by Chebyshev's inequality, it follows that $0 < \mu(E_k) \leq k^{-1} \|f\|_1$ for all $k \in \mathbb{N}$, and $\mu(E_k) \searrow 0$. Then if $T : G \curvearrowright (X, \mu)$ is the trivial action, i.e. $T_g = \operatorname{id}_X$ for all $g \in G$, then 
	\begin{align*}
		\alpha_{E_k} \left( \operatorname{Avg}_{F_k} |f| \right)	& = \alpha_{E_k} (|f|) \\
		& \geq k \\
		\Rightarrow \alpha_{E_k} \left( \operatorname{Avg}_{F_k} |f| \right)	& \stackrel{k \to \infty}{\to} + \infty .
	\end{align*}
	Based on this example, we can see that in contrast with Proposition \ref{Constant spatial averaging sequence}, there's no hope for improving Proposition \ref{Measure to 1} to even the case where $f \in L^{\infty-}(X, \mu) = \bigcap_{p \in [1, \infty)} L^p(X, \mu)$.
	
	\section{Temporo-spatial differentiation theorems around sets of rapidly vanishing diameter}\label{Temporo-spatial differentiation theorems around sets of rapidly vanishing diameter}
	
	In this section, we'll be concerned with the following general setup and question: Let $(X, p)$ be a compact pseudometric space, and let $T : G \curvearrowright X$ be a continuous action of a locally compact group $G$ on $X$ which preserves a Borel probability measure $\mu$ on $X$. Now fix some point $x_0 \in X$, and consider a net of positive-measure subsets $C_i$ of $X$ containing $x_0$. When will the temporo-spatial derivative relative to $C_i$ (and some averaging net $F_i$) resemble the pointwise temporal average at $x_0$? Theorem \ref{Non-Holder pointwise and temporo-spatial} establishes a powerful sufficient condition: If $f : X \to \mathbb{C}$ is uniformly continuous and bounded, and the diameter of the elements of the net $C_i$ go to $0$ sufficiently fast, then we'll have that $\left( \operatorname{Avg}_{F_i} f \right) (x_0) \approx \alpha_{C_i} \left( \operatorname{Avg}_{F_i} f \right)$, where "sufficiently fast" depends upon the (pseudo)metric properties of the continuous action, the averaging net $(F_i)_{i \in \mathscr{I}}$, and the point $x_0$. In this situation, we can reduce the temporo-spatial problem to a problem of taking a pointwise ergodic average. We then consider cases where narrowing our focus (e.g. considering H\"older actions instead of general continuous actions) allow us to improve the diameter decay rate. We then move on to make statements about the "probabilistically generic" behavior of these temporo-spatial derivatives by appealing to pointwise convergence results from ergodic theory. Finally, we extend this pointwise reduction to the setting of nonconventional ergodic averages with Theorem \ref{Pointwise reduction for multiple ergodic averages}.
	
	Several results in this section will be quite general in their statement, and as such will sometimes require a number of additional hypotheses that are satisfied automatically in many reasonable situations. We make notes after the proofs of some results to note that certain hypotheses stated explicitly in the results in question are satisfied a priori in certain reasonable cases.
	
	Our first result of this section describes a sufficient condition for the temporo-spatial averages to reduce to pointwise averages.
	
	\begin{Lem}\label{Reduction lemma}
		Let $(X, p)$ be a compact pseudometric space, and let $T : G \curvearrowright X$ be a continuous action of a locally compact topological group $G$ (not necessarily amenable) on $X$. Let $\mu$ be a regular Borel probability measure on $X$. Fix a point $x_0 \in X$.
		
		Let $(F_i)_{i \in \mathscr{I}}$ be a net of compact subsets of $G$ with positive Haar measure $m$. Let $(C_i)_{i \in \mathscr{I}}$ be a net of measurable subsets of $X$ such that $\mu(C_i) > 0$ and $x_0 \in C_i$ for all $i \in \mathscr{I}$. Suppose that for every $\delta > 0$, there exists a net $(A_i)_{i \in \mathscr{I}}$ of measurable subsets of $G$ such that
		\begin{align*}
			A_i	& \subseteq \left\{ g \in F_i : \operatorname{diam}(C_i) \leq \delta \right\} ,	& (\forall i \in \mathscr{I}) \\
			\lim_i	\frac{m(A_i)}{m(F_i)}	& = 1 .
		\end{align*}
		Let $f : X \to \mathbb{C}$ be a continuous function. Then
		$$
		\lim_i \left| \left( \operatorname{Avg}_{F_i} f \right) (x_0) - \alpha_{C_i} \left( \operatorname{Avg}_{F_i} f \right) \right| = 0 .
		$$
	\end{Lem}
	
	\begin{proof}
		Fix $\epsilon > 0$. Since $f$ is uniformly continuous (see \cite[Lemma 3.1]{Assani-Young}), there exists $\delta > 0$ such that if $y_1, y_2 \in X$, and $p(y_1, y_2) \leq \delta$, then $|f(y_1) - f(y_2)| \leq \frac{\epsilon}{2 \lambda}$. Let $(A_i)_{i \in \mathscr{I}}$ be as in the lemma statement, and set $B_i = F_i \setminus A_i$ for all $i \in \mathscr{I}$, so $\lim_i m(B_i) / m(F_i) = 0$.
		
		Now, we estimate
		\begin{align*}
			& \left| \left( \operatorname{Avg}_{F_i} f \right) (x_0) - \alpha_{C_i}(\operatorname{Avg}_{F_i} f) \right| \\
			=	& \left| \alpha_{C_i} \left( \left( \operatorname{Avg}_{F_i} f \right) (x_0) - \operatorname{Avg}_{F_i} f \right) \right| \\
			=	& \left| \alpha_{C_i} \left( \frac{1}{m(F_i)} \int_{F_i} (f(T_g x_0) - (f \circ T_g)) \right) \mathrm{d} m(g) \right| \\
			\leq	& \left| \alpha_{C_i} \left( \frac{1}{m(F_i)} \int_{A_i} (f(T_g x_0) - (f \circ T_g)) \mathrm{d} m(g) \right) \right| \\
			+	& \left| \alpha_{C_i} \left( \frac{1}{m(F_i)} \int_{B_i} (f(T_g x_0) - (f \circ T_g)) \mathrm{d} m(g) \right) \right|
		\end{align*}
		Our goal now is to bound both
		\begin{align*}
			& \left| \alpha_{C_i} \left( \frac{1}{m(F_i)} \int_{A_i} (f(T_g x_0) - (f \circ T_g)) \mathrm{d} m(g) \right) \right| , \\
			& \left| \alpha_{C_i} \left( \frac{1}{m(F_i)} \int_{B_i} (f(T_g x_0) - (f \circ T_g)) \mathrm{d} m(g) \right) \right|
		\end{align*}
		by $\frac{\epsilon}{2}$.
		
		First, we estimate the term $\left| \alpha_{C_i} \left( \frac{1}{m(F_i)} \int_{A_i} (f(T_g x_0) - (f \circ T_g)) \mathrm{d} m(g) \right) \right|$. We see that if $g \in A_i$, then
		\begin{align*}
			& \left| \alpha_{C_i} \left( \frac{1}{m(F_i)} \int_{A_i} (f(T_g x_0) - (f \circ T_g)) \mathrm{d} m(g) \right) \right| \\
			=	& \left| \frac{1}{\mu(C_i)} \int_{C_i} \frac{1}{m(F_i)} \int_{A_i} (f(T_g x_0) - f(T_g x)) \mathrm{d} m(g) \mathrm{d} \mu(x) \right| \\
			\leq	& \frac{1}{\mu(C_i)} \int_{C_i} \frac{1}{m(F_i)} \int_{A_i} \left| f(T_g x_0) - f(T_g x) \right| \mathrm{d} m(g) \mathrm{d} \mu(x)
		\end{align*}
		But if $y \in T_g C_i$, and $D_{x_0}(g, \operatorname{diam}(C_i)) \leq \delta$, then $x_0 , T_{g^{-1}} y \in C_i$, meaning that
		$$\rho(T_g x_0 , y) = \rho (T_g x_0 , T_g (T_{g^{-1}} y)) \leq D_{x_0} (g, T_{g^{-1}} y) \leq D_{x_0}(g, \operatorname{diam}(C_i)) \leq \delta .$$
		Therefore
		\begin{align*}
			\frac{1}{\mu(C_i)} \int_{C_i} \frac{1}{m(F_i)} \int_{A_i} \left| f(T_g x_0) - f(T_g x) \right| \mathrm{d} m(g) \mathrm{d} \mu(x)	& \leq \frac{1}{\mu(C_i)} \int_{C_i} \frac{1}{m(F_i)} \int_{A_i} \frac{\epsilon}{2} \mathrm{d} m(g) \mathrm{d} \mu(x) \\
			& = \frac{\epsilon}{2} .
		\end{align*}
		
		Now, we bound the term $\left| \alpha_{C_i} \left( \frac{1}{m(F_i)} \int_{B_i} (f(T_g x_0) - (f \circ T_g)) \mathrm{d} m(g) \right) \right|$. By estimates similar to those used to approximate the former term, we have that
		\begin{align*}
			& \left| \frac{1}{\mu(C_i)} \int_{C_i} \frac{1}{m(F_i)} \int_{B_i} (f(T_g x_0) - f(T_g x)) \mathrm{d} m(g) \mathrm{d} \mu(x) \right| \\
			\leq	& \frac{1}{\mu(C_i)} \int_{C_i} \frac{1}{m(F_i)} \int_{B_i} \left| f(T_g x_0) - f(T_g x) \right| \mathrm{d} m(g) \mathrm{d} \mu(x) \\
			\leq	& \frac{1}{\mu(C_i)} \int_{C_i} \frac{1}{m(F_i)} \int_{B_i} \left( 2 \|f\|_u \right) \mathrm{d} m(g) \mathrm{d} \mu(x)
		\end{align*}
		Choose $I \in \mathscr{I}$ such that if $i \geq I$, then $\frac{m(B_i)}{m(F_i)} \leq \frac{\epsilon}{4 \max \{1, \|f\|_u\}}$. Then if $i \geq I$, then $\frac{m(B_i)}{m(F_i)} (2 \|f\|_u) \leq \frac{\epsilon}{2}$.
		
		Therefore, if $i \geq I$, then $\left| \left( \operatorname{Avg}_{F_i} f \right) (x_0) - \alpha_{C_i}(\operatorname{Avg}_{F_i} f) \right| \leq \frac{\epsilon}{2} + \frac{\epsilon}{2} = \epsilon$.
	\end{proof}
	
	We have stated Lemma \ref{Reduction lemma} for \emph{pseudo}metric spaces, rather than just metric spaces. In \cite{Assani-Young}, we found that looking at certain pseudometric spaces helped us to establish convergence results for certain temporo-spatial averages. For example, a result like our Theorem \ref{Non-Holder pointwise and temporo-spatial} (\cite[Proposition 3.2]{Assani-Young}) was useful in proving \cite[Theorem 3.5]{Assani-Young}. For this reason, we state several results of this section in terms of compact pseudometric spaces.
	
	We also observe that Lemma \ref{Reduction lemma} does not assume that the action $T: G \curvearrowright (X, \mu)$ is measure-preserving, only continuous.
	
	Lemma \ref{Reduction lemma} as stated is a powerful tool for achieving the kind of reduction to the pointwise setting that we aim for, but we desire still a sufficient condition for the hypotheses of the lemma to attain. The following result states that, under appropriate conditions, we can find an $x_0$-dependent diameter decay condition on $(C_i)_{i \in \mathscr{I}}$ for this reduction to attain.
	
	\begin{Thm}\label{Non-Holder pointwise and temporo-spatial}
		Let $(X, p)$ be a compact pseudometric space, and let $T : G \curvearrowright X$ be a continuous action of a locally compact topological group $G$ (not necessarily amenable) on $X$. Let $\mu$ be a regular Borel probability measure on $X$. Fix a point $x_0 \in X$, and for each $g \in G, r \in (0, \infty)$, let $D_{x_0}(g, r)$ be the value
		$$D_{x_0}(g, r) = \sup \left\{ p(T_g x_0, T_g x) : x \in X , p(x_0, x) \leq r \right\} ,$$
		and assume that $D_{x_0} (\cdot, r) : G \to (0, \infty)$ is measurable for each $r \in (0, \infty)$.
		
		Let $(F_i)_{i \in \mathscr{I}}$ be a net of compact subsets of $G$ with positive Haar measure $m$. Let $(C_i)_{i \in \mathscr{I}}$ be a net of measurable subsets of $X$ such that $\mu(C_i) > 0$ and $x_0 \in C_i$ for all $i \in \mathscr{I}$. Suppose that for every $\delta > 0$, we have
		$$
		\lim_i \frac{m \left( \left\{ g \in F_i : D_{x_0}(g, \operatorname{diam}(C_i)) > \delta \right\} \right)}{m(F_i)} = 0 .
		$$
		Let $f : X \to \mathbb{C}$ be a continuous function. Then
		$$
		\lim_i \left| \left( \operatorname{Avg}_{F_i} f \right) (x_0) - \alpha_{C_i} \left( \operatorname{Avg}_{F_i} f \right) \right| = 0 .
		$$
	\end{Thm}
	
	\begin{proof}
		For each $\delta > 0$, set
		$$
		A_i	= \left\{ g \in F_i : D_{x_0}(g, \operatorname{diam}(C_i)) \leq \delta \right\} .
		$$
		The result follows from Lemma \ref{Reduction lemma}.
	\end{proof}
	
	The assumption that $D_{x_0} (\cdot, r) : G \to (0, \infty)$ be a measurable function in $G$ for all $r \in (0, \infty)$, though relevant to make sure the sets $A_i$ in our proof are measurable, is satisfied automatically in the case where $G$ is discrete. Our use of this function $D_{x_0}$ ensures that the condition being imposed is in fact a decay condition on $\operatorname{diam}(C_i)$, in the sense that if $\left( C_i \right)_{i \in \mathscr{I}}$ is a net satisfying the condition
	\begin{align*}
		\lim_i \frac{m \left( \left\{ g \in F_i : D_{x_0}(g, \operatorname{diam}(C_i)) > \delta \right\} \right)}{m(F_i)}	& = 0	& (\forall \delta > 0) , 
	\end{align*}
	and $\left( C_i ' \right)_{i \in \mathscr{I}}$ is a net of measurable subsets of $X$ containing $x_0$ with positive measure, and $\operatorname{diam} \left( C_i ' \right) \leq \operatorname{diam}(C_i)$ for all $i \in \mathscr{I}$, then $\left( C_i' \right)_{i \in \mathscr{I}}$ will also satisfy the condition. However, this decay rate depends on $x_0$, a shortcoming which can be overcome with some additional conditions on the action $T$, as will be seen in Theorem \ref{Pointwise and temporo-spatial}.
	
	\begin{Def}
		Let $(X, p)$ be a pseudometric space (not necessarily compact), and let $T : G \curvearrowright X$ be an action of a group $G$ on $X$. We call the action \emph{H\"older} if for every $g \in G$ exist $H(g), L(g) \in (0, \infty)$ such that
		\begin{align*}
			p \left( T_g x , T_g y \right)	& \leq L(g) \cdot p(x, y)^{H(g)}	& (\forall g \in G , x \in X , y \in X) .
		\end{align*} 
	\end{Def}
	
	Our next result shows that if we assume that our action is H\"{o}lder, and the H\"older parameters of $T_g$ satisfy certain measurability properties as functions of $G$, then this diameter decay rate can be chosen independent of $x_0$. We remark now that our statement of the result is quite wordy, with several hypotheses, but as we'll explain shortly, several of these hypotheses are satisfied automatically in many cases.
	
	\begin{Thm}\label{Pointwise and temporo-spatial}
		Let $(X, p)$ be a compact pseudometric space, and let $T : G \curvearrowright X$ be a continuous action of a locally compact topological group $G$ (not necessarily amenable) on $X$. Let $\mu$ be a regular Borel probability measure on $X$. Assume further that there exist measurable functions $H , L : G \to (0, \infty)$ such that
		\begin{align*}
			p \left( T_g x , T_g y \right)	& \leq L(g) \cdot p(x, y)^{H(g)}	& (\forall g \in G , x \in X , y \in X) .
		\end{align*}
		Let $(F_i)_{i \in \mathscr{I}}$ be a net of compact subsets of $G$ with positive Haar measure $m$. Let $(C_i)_{i \in \mathscr{I}}$ be a net of measurable subsets of $X$ such that $\mu(C_i) > 0$ and $x \in C_i$ for all $i \in \mathscr{I}$. Suppose that for every $\delta > 0$, we have
		$$
		\lim_i \frac{m \left( \left\{ g \in F_i : L(g) \cdot \operatorname{diam}(C_i)^{H(g)} > \delta \right\} \right)}{m(F_i)} = 0 .
		$$
		Let $x_0 \in X$ be a point in $X$, and let $f : X \to \mathbb{C}$ be a uniformly bounded continuous function. Then
		$$
		\lim_i \left| \left( \operatorname{Avg}_{F_i} f \right) (x_0) - \alpha_{C_i} \left( \operatorname{Avg}_{F_i} f \right) \right| = 0 .
		$$
	\end{Thm}
	
	\begin{proof}
		We first observe that if $p(T_g x , T_g y) \leq L(g) p(x, y)^{H(g)}$, then $D_{x_0}(g, r) \leq L(g) r^{H(g)}$ for all $x_0 \in X$, so $\operatorname{diam}(T_g C_i) \leq L(g) \cdot \operatorname{diam}(C_i)^{H(g)}$. Given $\delta > 0$, set
		$$A_i = \left\{ g \in F_i : L(g) \cdot \operatorname{diam}(C_i)^{H(g)} \leq \delta \right\},$$
		and apply Lemma \ref{Reduction lemma}.
	\end{proof}
	
	\begin{Rmk}\label{Comments on Pointwise and t-s}
		\begin{itemize}
			\item If $G$ is discrete, then the measurability assumptions on $H, L$ are automatically fulfilled.
			\item If $T_g$ is Lipschitz for all $g \in G$, then we can take $H$ to be the constant function $1$. This is the case in particular if $T$ is an action on a compact Riemannian manifold $X$ by diffeomorphisms.
			\item In the special case where $G = \mathbb{Z}$, if both $T_1, T_{-1}$ are H\"older with exponent $\alpha_0$ and coefficient $L_0$, then for $n \geq 0$, we can take $H(n) = \alpha_0^{|n|} , L(n) = L_0^{|n|}$.
			\item If $G$ acts by isometries, then we can take $H, L$ to both be the constants $1$.
		\end{itemize}
	\end{Rmk}
	
	Theorem \ref{Pointwise and temporo-spatial} says that given a H\"older action $T$ of a group $G$ (subject to certain measurability conditions) on a compact pseudometric probability space, and an averaging net $(F_i)_{i \in \mathscr{I}}$, there exists a diameter decay rate such that if $(C_i)_{i \in \mathscr{I}}$ is a net of positive-measure sets containing a fixed point $x_0$, then the temporo-spatial derivative at $C_i$ will resemble the temporal pointwise average. Notably, this decay rate depends only on the averaging net and the H\"older condition on $T$, and not on the point $x_0$ or the function $f$.
	
	Theorem \ref{Pointwise and temporo-spatial} cannot be called sharp in the strictest sense, since given any net $(C_i)$ satisfying the hypotheses of Theorem \ref{Pointwise and temporo-spatial}, we could replace all the $C_i$ with $C_i \cup E$, where $E$ is some fixed subset of $X$ with positive diameter but measure $0$. A truly sharp Theorem \ref{Pointwise and temporo-spatial} would -at the very least- have to account for a notion of ``essential diameter."
	
	Under an additional assumption on the function being averaged, we can provide quantitative estimates on the approximation in Theorem \ref{Pointwise and temporo-spatial}.
	
	\begin{Prop}\label{Quantitative pointwise and temporo-spatial}
		Let $(X, p)$ be a compact pseudometric space, and let $T : G \curvearrowright X$ be a continuous action of a locally compact topological group $G$ (not necessarily amenable) on $X$. Let $\mu$ be a regular Borel probability measure on $X$. Assume further that there exist measurable functions $H , L : G \to (0, \infty)$ such that
		\begin{align*}
			p \left( T_g x , T_g y \right)	& \leq L(g) \cdot p(x, y)^{H(g)}	& (\forall g \in G , x \in X , y \in X) .
		\end{align*}
		Let $F$ be a compact subset of $G$ with positive Haar measure $m$, and let $C$ be a measurable subset of $X$ such that $\mu(C) > 0$.
		
		Let $x_0 \in X$ be a point in $X$, and let $f : X \to \mathbb{C}$ be a H\"older function with constants $c, \beta$ for which
		\begin{align*}
			|f(x) - f(y)|	& \leq c \cdot \rho(x, y)^\beta	& (\forall x, y \in X) .
		\end{align*}
		Then
		$$
		\left| \left( \operatorname{Avg}_{F} f \right) (x_0) - \alpha_{C} \left( \operatorname{Avg}_{F} f \right) \right| \leq \frac{c}{m(F)} \int_F L(g)^\beta \cdot \operatorname{diam}(C)^{\beta H(g)} \mathrm{d} m(g) .
		$$
	\end{Prop}
	
	\begin{proof}
		\begin{align*}
			& \left| \left( \operatorname{Avg}_{F} f \right) (x_0) - \alpha_{C} \left( \operatorname{Avg}_{F} f \right) \right| \\
			=	& \left| \alpha_{C} \left( \frac{1}{m(F)} \int_{F} \left(f\left(T_g x_0\right) - \left(f \circ T_g\right)\right) \mathrm{d} m(g) \right) \right| \\
			=	& \left| \frac{1}{m(C)} \int_C \frac{1}{m(F)} \int_{F} \left(f(T_g x_0) - (f(T_g x)) \right) \mathrm{d} m(g) \mathrm{d} \mu(x) \right| \\
			\leq	& \frac{1}{m(C)} \int_C \frac{1}{m(F)} \int_{F} \left| f\left(T_g x_0\right) - f\left(T_g x\right) \right| \mathrm{d} m(g) \mathrm{d} \mu(x) \\
			\leq	& \frac{1}{m(C)} \int_C \frac{1}{m(F)} \int_{F} c \cdot \rho \left( T_g x_0, T_g x \right)^\beta \mathrm{d} m(g) \mathrm{d} \mu(x) \\
			\leq	& \frac{1}{m(C)} \int_C \frac{1}{m(F)} \int_{F} c \cdot \left( L(g) \cdot \rho \left( x_0, x \right)^{H(g)} \right)^\beta \mathrm{d} m(g) \mathrm{d} \mu(x) \\
			\leq	& \frac{1}{m(C)} \int_C \frac{1}{m(F)} \int_{F} c \cdot \left( L(g) \cdot \operatorname{diam}(C)^{H(g)} \right)^\beta \mathrm{d} m(g) \mathrm{d} \mu(x) \\
			=	& c \frac{1}{m(C)} \int_C \frac{1}{m(F)} \int_{F} L(g)^\beta \cdot \operatorname{diam}(C)^{\beta H(g)} \mathrm{d} m(g) \mathrm{d} \mu(x) \\
			=	& \frac{c}{m(F)} \int_F L(g)^\beta \cdot \operatorname{diam}(C)^{\beta H(g)} \mathrm{d} m(g)
		\end{align*}
	\end{proof}
	
	Our next result takes us in the direction of a ``random temporo-spatial differentiation problem," where we consider a temporo-spatial problem in which the spatial averaging net is considered to be chosen ``randomly" according to some scheme or constraints.
	
	\begin{Cor}\label{Almost-sure convergence}
		Let $T : G \curvearrowright X$ be a continuous action of a locally compact topological group $G$ on a compact pseudometric space $X = (X, p)$, and let $\mu$ be a regular Borel probability measure on $X$, and let $(F_i)_{i \in \mathscr{I}}$ be a net in $G$. Let $H, L : G \to (0, \infty)$ be measurable functions such that
		\begin{align*}
			p \left( T_g x , T_g y \right)	& \leq L(g) \cdot p(x, y)^{H(g)}	& (\forall g \in G , x \in X , y \in X) .
		\end{align*}
		Suppose that for each $x \in X$, the net $(C_i(x))_{i \in \mathscr{I}}$ is a net of measurable subsets $C_i(x)$ of $X$ containing the point $x$ such that $\mu(C_i(x)) > 0$ for all $x \in X$, as well as that for almost all $x \in X$, we have
		$$\lim_i \frac{m \left( \left\{ g \in F_i : L(g) \cdot \operatorname{diam}(C_i(x))^{H(g)} > \delta \right\} \right)}{m(F_i)} = 0$$
		for all $\delta > 0$. Let $f : X \to \mathbb{C}$ be a continuous function, and suppose that for almost all $x \in X$, we have that $\lim_{i} \operatorname{Avg}_{F_i} f(x) = f^*(x)$, where $f^*$ is a measurable function $X \to \mathbb{C}$. Then
		$$\lim_i \alpha_{C_i(x)} \left( \operatorname{Avg}_{F_i} f \right) = f^*(x)$$
		for almost all $x \in X$.
	\end{Cor}
	
	\begin{Rmk}\label{Sufficient conditions for almost-sure convergence}
		Corollary \ref{Almost-sure convergence} is a tool that turns almost-sure pointwise convergence results from ergodic theory into almost-sure convergence results for classes of random temporo-spatial differentiations. Corollaries \ref{Lindenstrauss random TSD} and \ref{Bourgain random TSD}, corresponding to the Lindenstrauss pointwise ergodic theorem and Bourgain's theorem on pointwise convergence of averages along polynomials, respectively, are special cases of Corollary \ref{Almost-sure convergence}. In principle, there is a special case of Corollary \ref{Almost-sure convergence} corresponding to any result that ensures the almost-sure pointwise convergence of an ergodic average.
	\end{Rmk}
	
	\begin{proof}[Proof of Corollary \ref{Almost-sure convergence}]
		Let
		\begin{align*}
			A	& = \bigcap_{k = 1}^\infty \left\{ x\in X : \lim_i \frac{m \left( \left\{ g \in F_i : L(g) \cdot \operatorname{diam}(C_i(x))^{H(g)} > 1/k \right\} \right)}{m(F_i)} = 0 \right\}, \\
			B	& = \left\{ x \in X : \lim_i \operatorname{Avg}_{F_i} f(x) = f^*(x) \right\} .
		\end{align*}
		Both $A, B$ are of full measure by hypothesis, and thus so is $A \cap B$. Let $x \in A \cap B$. Then
		\begin{align*}
			\left| \alpha_{C_i(x)} \left( \operatorname{Avg}_{F_i} f \right) - f^*(x) \right|	& \leq \left| \alpha_{C_i(x)} \left( \operatorname{Avg}_{F_i} f \right) - \operatorname{Avg}_{F_i} f(x) \right| + \left| \operatorname{Avg}_{F_i} f(x) - f^* (x) \right| \\
			& \stackrel{i \to \infty}{\to} 0 ,
		\end{align*}
		where the first summand goes to $0$ (by Theorem \ref{Pointwise and temporo-spatial}) because $x \in A$ and the second summand goes to $0$ because $x \in B$.
	\end{proof}
	
	As a rule, results like Theorem \ref{Pointwise and temporo-spatial} lead naturally to results like Corollary \ref{Almost-sure convergence}, and we'll see several other examples of this in this article. Theorem \ref{Pointwise and temporo-spatial} provides a sufficient condition for a spatial averaging net $(C_i)_{i \in \mathscr{I}}$ around a point $x$ to induce a temporo-spatial differentiation problem that's reducible to a pointwise temporal problem at that point $x$; it then follows that if we have some scheme for associating to every point $x$ a spatial averaging net $(C_i(x))_{i \in \mathscr{I}}$ around $x$, and we know that $\operatorname{Avg}_{F_i} f(x) \to f^*(x)$ almost surely for $f \in C(X)$ continuous, then we have a convergence result for the ``random temporo-spatial differentiation problem" $\left(\alpha_{C_i(x)} \left( \operatorname{Avg}_{F_i} f \right) \right)_{i \in \mathscr{I}}$. There will be several other examples of results like Corollary \ref{Almost-sure convergence} in various contexts, taking some temporal pointwise reduction result like Theorem \ref{Pointwise and temporo-spatial} and extrapolating a statement about random temporo-spatial problems.
	
	It should be noted, however, that the convergence in Corollary \ref{Almost-sure convergence} will in general be only for almost every $x \in X$, rather than all $x \in X$. If there exists a point $x \in X$ where $\left( \operatorname{Avg}_{F_i} f(x) \right)_{i \in \mathscr{I}}$ does not converge to $f^*(x)$, then Theorem \ref{Pointwise and temporo-spatial} tells us that $\left( \alpha_{C_i(x)} \left( \operatorname{Avg}_{F_i} f \right) \right)_{i \in \mathscr{I}}$ won't either.
	
	\begin{Cor}\label{Probabilistic genericity}
		Let $T : G \curvearrowright X$ be an action of a locally compact topological group $G$ on a compact metric space $X = (X, \rho)$ that preserves a Borel probability measure $\mu$ on $X$, and let $(F_i)_{i \in \mathscr{I}}$ be a net. Let $H, L : G \to (0, \infty)$ be measurable functions such that
		\begin{align*}
			p \left( T_g x , T_g y \right)	& \leq L(g) \cdot p(x, y)^{H(g)}	& (\forall g \in G , x \in X , y \in X) .
		\end{align*}
		Suppose that for each $x \in X$, the net $(C_i(x))_{i \in \mathscr{I}}$ is a net of measurable subsets $C_i(x)$ of $X$ containing the point $x$ such that $\mu(C_i(x)) > 0$ for all $x \in X$, and
		$$\lim_i \frac{m \left( \left\{ g \in F_i : L(g) \cdot \operatorname{diam}(C_i(x))^{H(g)} > \delta \right\} \right)}{m(F_i)} = 0$$
		for almost all $x \in X$. Suppose that for $\mu$-almost all $x \in X$, we have
		\begin{align*}
			\lim_i \operatorname{Avg}_{F_i} f(x)	& = \int f \mathrm{d} \mu	& \left( \forall f \in C(X) \right) 
		\end{align*}
		Then for almost all $x \in X$, we have
		$$\lim_i \alpha_{C_i(x)} \left( \operatorname{Avg}_{F_i} f \right) = \int f \mathrm{d} \mu .$$
	\end{Cor}
	
	\begin{proof}
		Since $X$ is compact metrizable, it follows that $C(X)$ is separable, so let $\left\{ f_n : n \in \mathbb{N} \right\}$ be a subset of $C(X)$ with dense span. For each $n \in \mathbb{N}$, set
		\begin{align*}
			A_n	& = \left\{ x \in X : \lim_i \alpha_{C_i(x)}\left( \operatorname{Avg}_{F_i} f_n \right) = \int f \mathrm{d} \mu \right\} .
		\end{align*}
		Each $A_n$ is of full measure.
		
		Let $f \in C(X)$, and fix $N \in \mathbb{N}$. Choose $J_N \in \mathbb{N}$ and a sequence $z_{1, N}, \ldots, z_{J_N, N} \in \mathbb{C}$ such that
		$$
		\lim_{N \to \infty} \left\| f - \sum_{j = 1}^{J_N} z_{j, N} f_j \right\|_{C(X)} \leq \frac{1}{3 N} .
		$$
		For convenience, set $\phi_N = \sum_{j = 1}^{J_N} z_{j, N} f_j$. Then $\left\| \int f \mathrm{d} \mu - \int \phi_N \mathrm{d} \mu \right\|_{L^\infty(X, \mu)} \leq \| f - \phi_N \|_{C(X)} \leq \frac{\epsilon}{3}$.
		
		Now for $j \in \{1, \ldots, J_N\}$, choose $i_{j, N} \in \mathscr{I}$ such that if $i \geq i_{j, N}$, then
		$$
		\left| \alpha_{C_i}\left( \operatorname{Avg}_{F_i} f_j(x) \right) - \int f \mathrm{d} \mu \right| \leq \frac{1}{3 N^2 \max \{ z_{j, N} , 1 \}} .
		$$
		Choose $I_N \in \mathscr{I}$ such that $I_N \geq i_{j, N}$ for all $j \in \{1, \ldots, J_N\}$, and let $x \in \bigcap_{n \in \mathbb{N}} A_n$. Then if $i \geq I_N$, we have
		\begin{align*}
			\left| \alpha_{C_i} \left( \operatorname{Avg}_{F_i} f \right) - \int f \mathrm{d} \mu \right|	& \leq \left| \alpha_{C_i} \left( \operatorname{Avg}_{F_i} f \right) - \alpha_{C_i} \left( \operatorname{Avg}_{F_i} \phi_N \right) \right| \\
			& + \left| \alpha_{C_i} \left( \operatorname{Avg}_{F_i} \phi_N \right) - \int \phi_N \mathrm{d} \mu \right| \\
			& + \left| \int \phi_N \mathrm{d} \mu - \int f \mathrm{d} \mu \right| .
		\end{align*}
		We bound each of the three summands by $\frac{1}{3 N}$ in turn.
		Firstly, we can see that
		$$
		\left| \alpha_{C_i} \left( \operatorname{Avg}_{F_i} f \right) - \alpha_{C_i} \left( \operatorname{Avg}_{F_i} \phi_N \right) \right| = \left| \alpha_{C_i} \left( \operatorname{Avg}_{F_i} \left( f - \phi_N \right) \right) \right| \leq \left\| f - \phi_N \right\|_{C(X)} \leq \frac{1}{3 N} ,
		$$
		which addresses the first summand. For the second summand, we have
		\begin{align*}
			\left| \alpha_{C_i} \left( \operatorname{Avg}_{F_i} \phi_N \right) - \int \phi_N \mathrm{d} \mu \right|	& = \left| \sum_{j = 1}^{J_N} z_{j, N} \left( \alpha_{C_i} \left( \operatorname{Avg}_{F_i} f_j \right) - \int f_j \mathrm{d} \mu \right) \right| \\
			& \leq \sum_{j = 1}^N |z_{j, N}| \cdot \left| \alpha_{C_i} \left( \operatorname{Avg}_{F_i} f_j \right) - \int f_j \mathrm{d} \mu \right| \\
			& \leq \sum_{j = 1}^N |z_{j, N}| \frac{1}{3 N^2 \max\{ |z_{j, N}| , 1\}} \\
			& \leq \frac{1}{3 N} .
		\end{align*}
		Finally, for the third summand, we have that
		$$\left| \int \phi_N \mathrm{d} \mu - \int f \mathrm{d} \mu \right| \leq \| \phi_N - f \|_{C(X)} \leq \frac{1}{3 N} .$$
		Taken together, these tell us that for every $N \in \mathbb{N}, x \in \bigcap_{n \in \mathbb{N}} A_n$, there exists $I \in \mathscr{I}$ such that if $i \geq I$, then $\left| \alpha_{C_i} \left( \operatorname{Avg}_{F_i} f \right) - \int f \mathrm{d} \mu \right| \leq \frac{1}{N}$. Therefore $\lim_{i} \alpha_{C_i} \left( \operatorname{Avg}_{F_i} f \right) = \int f \mathrm{d} \mu$ for all $x \in \bigcap_{n \in \mathbb{N}} A_n$. Since each $A_n$ is of full measure, it follows that their countable intersection $\bigcap_{n = 1}^\infty A_n$ is of full measure, yielding our desired almost-sure convergence.
	\end{proof}
	
	This result tells us that if to (almost) every $x \in X$ we assign a net $(C_i(x))$ of sets of positive measure with rapidly decaying diameter, and $\operatorname{Avg}_{F_i} f \to \int f \mathrm{d} \mu$ for all $f \in C(X)$ then the "probabilistically generic" behavior is that $\alpha_{C_i} \left( \operatorname{Avg}_{F_i} f \right) \to \int f \mathrm{d} \mu$.
	
	Corollary \ref{Probabilistic genericity} encompasses several results from \cite{Assani-Young}, including Theorem  2.1, Theorem 2.3, and Corollary 3.3. Proposition 3.2 from that paper can also be recovered from our Theorem \ref{Almost-sure convergence}. Corollary \ref{Probabilistic genericity} is motivated by the desire to find positive convergence results for temporo-spatial differentiations relative to actions of groups other than $\mathbb{Z}$, as well as to find to find positive convergence results for temporo-spatial differentiations relative to averages over other subsequences of $\mathbb{Z}$. Moreover, Corollary \ref{Probabilistic genericity} opens the door to temporo-spatial differentiations along subsequences. We present here a few examples.
	
	\begin{Cor}\label{Lindenstrauss random TSD}
		Let $T : G \curvearrowright X$ be an action of a locally compact amenable topological group $G$ on a compact pseudometric space $X = (X, p)$  that preserves a regular Borel probability measure $\mu$ on $X$, and let $(F_k)_{k \in \mathbb{N}}$ be a tempered \Folner \space sequence for $G$. Let $H, L : G \to (0, \infty)$ be measurable functions such that
		\begin{align*}
			p \left( T_g x , T_g y \right)	& \leq L(g) \cdot p(x, y)^{H(g)}	& (\forall g \in G , x \in X , y \in X) .
		\end{align*}
		Suppose that for each $x \in X$, the sequence $(C_k(x))_{k \in \mathbb{N}}$ is a sequence of measurable subsets $C_k(x)$ of $X$ containing the point $x$ such that $\mu(C_k(x)) > 0$ for all $x \in X$, and
		$$\lim_{k \to \infty} \frac{m \left( \left\{ g \in F_k : L(g) \cdot \operatorname{diam}(C_k(x))^{H(g)} > \delta \right\} \right)}{m(F_k)} = 0$$
		for almost all $x \in X$.
		
		Then given $f \in C(X)$, for almost all $x \in X$, we have
		$$\lim_{k \to \infty} \alpha_{C_k(x)} \left( \operatorname{Avg}_{F_k} f \right) = \mathbb{E} f (x),$$
		where $\mathbb{E}$ is the projection onto the space of $T$-invariant functions in $L^\infty(X, \mu)$.
	\end{Cor}
	
	\begin{proof}
		The Lindenstrauss Ergodic Theorem \cite[Theorem 3.3]{LindenstraussErgodicTheorem} tells us that $\operatorname{Avg}_{F_k} f \to \mathbb{E} f$ almost surely, so we can apply Corollary \ref{Probabilistic genericity}.
	\end{proof}
	
	\begin{Cor}\label{Bourgain random TSD}
		Let $P \in \mathbb{R}[t]$ be a polynomial with real coefficients, and let $T : \mathbb{Z} \curvearrowright X$ be an action of $\mathbb{Z}$ on a compact pseudometric space $X = (X, p)$ that preserves a regular Borel probability measure $\mu$ on $X$. Let $F_k = \left\{ \lfloor P (1) \rfloor , \lfloor P (2) \rfloor , \ldots, \lfloor P (k) \rfloor \right\}$ for all $k \in \mathbb{N}$, and let $H, L : G \to (0, \infty)$ be functions such that
		\begin{align*}
			p \left( T_n x , T_n y \right)	& \leq L(n) \cdot p(x, y)^{H(n)}	& (\forall n \in \mathbb{Z} , x \in X , y \in X) .
		\end{align*}
		Suppose that for each $x \in X$, the sequence $(C_k(x))_{k \in \mathbb{N}}$ is a sequence of measurable subsets $C_k(x)$ of $X$ containing the point $x$ such that $\mu(C_k(x)) > 0$ for all $x \in X$, and
		$$\lim_{k \to \infty} \frac{m \left( \left\{ n \in F_k : L(n) \cdot \operatorname{diam}(C_k(x))^{H(n)} > \delta \right\} \right)}{m(F_k)} = 0$$
		for almost all $x \in X$. Let $f \in C(X)$. Then there exists a function $f^* \in L^\infty (X, \mu)$ such that for almost all $x \in X$, we have
		$$\lim_{k \to \infty} \alpha_{C_k(x)} \left( \operatorname{Avg}_{F_k} f \right) = f^* (x) .$$
	\end{Cor}
	
	\begin{proof}
		By \cite[Theorem 2]{BourgainPolynomials}, there exists $f^* \in L^\infty(X, \mu)$ such that $\operatorname{Avg}_{F_k} f(x) \to f^*(x)$ almost surely. Apply Corollary \ref{Almost-sure convergence}.
	\end{proof}
	
	Finally, we remark that a form of the pointwise reduction in Theorem \ref{Non-Holder pointwise and temporo-spatial} can be recovered in the context of nonconventional ergodic averages. In order to make the statement of this result a bit more readable, we use slightly different notation for the remainder of this section than we used in previous parts of this article, using $T_\ell$ to refer to an $\ell$th homeomorphism, rather than an action of the integer $\ell \in \mathbb{Z}$.
	
	\begin{Thm}\label{Pointwise reduction for multiple ergodic averages}
		Let $(X, p)$ be a compact pseudometric space, and let $T_1, \ldots, T_L$ be a family of homeomorphisms $T_\ell : X \to X$. Let $\mu$ be a regular Borel probability measure on $X$ invariant under each $T_\ell$. Let $\left( n_j^{(1)} \right)_{j = 1}^\infty , \ldots, \left( n_j^{(L)} \right)_{j = 1}^\infty$ be sequences of integers.
		
		Fix a point $x_0 \in X$, and let $(C_k)_{k = 1}^\infty$ be a sequence of measurable subsets of $X$ with positive measure for which $x_0 \in C_k$ and suppose that for each $\ell = 1, \ldots, \ell$, and every $\delta \in (0, \infty)$, we have that
		$$\frac{ \left| \left\{ j \in \{0, 1, \ldots, k - 1\} : \operatorname{diam} \left( T_\ell^{n_j^{(\ell)}} C_k \right) \geq \delta \right\} \right| }{k} \to 0 . $$
		
		Let $f_0, f_1, \ldots, f_L \in C(X)$. Then
		$$
		\lim_{k \to \infty} \left| \left( \frac{1}{k} \sum_{j = 0}^{k - 1} f_0(x_0) \prod_{\ell = 1}^L T_\ell^{n_j^{(\ell)}} f_\ell(x_0) \right) - \alpha_{C_k} \left( \frac{1}{k} \sum_{j = 0}^{k - 1} f_0 \prod_{\ell = 1}^L T_\ell^{n_j^{(\ell)}} f_\ell \right) \right| = 0 .
		$$
	\end{Thm}
	
	\begin{proof}
		For the sake of making some notation in this proof more concise, we'll write
		\begin{align*}
			T_0	& = \operatorname{id}_X , \\
			n_j^{(0)}	& = 1	& (\forall j \geq 0) ,
		\end{align*}
		meaning that $f_0 \prod_{\ell = 1}^L T_\ell^{n_j^{(\ell)}} f_\ell = \prod_{\ell = 0}^L T_\ell^{n_j^{(\ell)}} f_\ell$.
		We also use $\| \cdot \|_u$ to denote the uniform norm on $C(X)$.
		
		Fix $M = \max \left\{ 1 , \|f_0\|_u, \| f_1\|_u, \ldots , \|f_L\|_u \right\}$, and fix $\epsilon > 0$. By appealing to the uniform continuity of the functions $f_0, f_1, \ldots, f_L$, choose $\delta_0, \delta_1, \ldots, \delta_L > 0$ such that
		\begin{align*}
			\forall x \in X \; \forall y \in X \;	& \left[ (p(x, y) \leq \delta_\ell) \Rightarrow \left(|f_\ell(x) - f_\ell(y)| \leq \frac{\epsilon}{2 (L + 1) M^L} \right) \right]	& (\ell = 0, 1, \ldots, L) .
		\end{align*}
		Set $\delta = \min \left\{ \delta_0, \delta_1, \ldots, \delta_L \right\} > 0$, and set
		\begin{align*}
			A_k^{(\ell)}	& = \left\{ j \in \{0, 1, \ldots, k - 1\} : \operatorname{diam} \left( T_\ell^{n_j^{(\ell)}} C_k \right) < \delta \right\}	& (\ell = 1, \ldots, L, k \in \mathbb{N}) , \\
			A_k	& = \bigcap_{\ell = 1}^L A_k^{(\ell)} .
		\end{align*}
		Then
		\begin{align*}
			\frac{\left| \{0, 1, \ldots, k - 1 \} \setminus A_k \right|}{k}	& = \frac{ \left| \bigcup_{\ell = 1}^L \left( \{0, 1, \ldots, k - 1 \} \setminus A_k^{(L)} \right) \right| }{k} \\
			& \leq \sum_{\ell = 1}^L \frac{ \left| \left( \{0, 1, \ldots, k - 1 \} \setminus A_k^{(L)} \right) \right| }{k} \\
			& = \sum_{\ell = 1}^L \frac{ \left| \left\{ j \in \{0, 1, \ldots, k - 1\} : \operatorname{diam} \left( T_\ell^{n_j^{(\ell)}} C_k \right) \geq \delta \right\} \right| }{k} \\
			& \stackrel{k \to \infty}{\to} 0 .
		\end{align*}
		
		We now turn to estimating
		\begin{align*}
			& \left| \left( \sum_{j = 0}^{k - 1} f_0(x_0) \prod_{\ell = 1}^\ell T_\ell^{n_j^{(\ell)}} f_\ell(x_0) \right) - \alpha_{C_k} \left( \sum_{j = 0}^{k - 1} f_0 \prod_{\ell = 1}^\ell T_\ell^{n_j^{(\ell)}} f_\ell \right) \right| \\
			\leq	& \frac{1}{k} \sum_{j = 0}^{k - 1} \left| \left( f_0(x_0) \prod_{\ell = 1}^\ell T_\ell^{n_j^{(\ell)}} f_\ell(x_0) \right) - \alpha_{C_k} \left( f_0 \prod_{\ell = 1}^\ell T_\ell^{n_j^{(\ell)}} f_\ell \right) \right| \\
			=	& \frac{1}{k} \sum_{j = 0}^{k - 1} \left| \alpha_{C_k} \left( \left( f_0(x_0) \prod_{\ell = 1}^\ell T_\ell^{n_j^{(\ell)}} f_\ell(x_0) \right) - f_0 \prod_{\ell = 1}^\ell T_\ell^{n_j^{(\ell)}} f_\ell \right) \right| \\
			=	& \frac{1}{k} \left[ \sum_{j \in A_k} \left| \alpha_{C_k} \left( \left( f_0(x_0) \prod_{\ell = 1}^\ell T_\ell^{n_j^{(\ell)}} f_\ell(x_0) \right) - f_0 \prod_{\ell = 1}^\ell T_\ell^{n_j^{(\ell)}} f_\ell \right) \right| \right] \\
			& + \frac{1}{k} \left[ \sum_{j \in \{0, 1, \ldots, k - 1\} \setminus A_k} \left| \alpha_{C_k} \left( \left( f_0(x_0) \prod_{\ell = 1}^\ell T_\ell^{n_j^{(\ell)}} f_\ell(x_0) \right) - f_0 \prod_{\ell = 1}^\ell T_\ell^{n_j^{(\ell)}} f_\ell \right) \right| \right] .
		\end{align*}
		In light of this decomposition, we make separate estimates on
		\begin{align*}
			\left| \alpha_{C_k} \left( \left( f_0(x_0) \prod_{\ell = 1}^\ell T_\ell^{n_j^{(\ell)}} f_\ell(x_0) \right) - f_0 \prod_{\ell = 1}^\ell T_\ell^{n_j^{(\ell)}} f_\ell \right) \right|
		\end{align*}
		based on whether $j \in A_k$ or $j \in \{0, 1, \ldots, k - 1\} \setminus A_k$.
		
		If $j \in A_k$, and $x \in A_k$, then $p \left( T_\ell^{n_j^{(\ell)}} x , T_\ell^{n_j^{(\ell)}} x_0 \right) < \delta$. Using an elementary ``telescoping" estimate, it follows that
		\begin{align*}
			& \left| \left( f_0(x) \prod_{\ell = 1}^\ell T_\ell^{n_j^{(\ell)}} f_\ell(x) \right) - \left( f_0(x_0) \prod_{\ell = 1}^\ell T_\ell^{n_j^{(\ell)}} f_\ell(x_0) \right) \right|	\\
			\leq	& \sum_{h = 0}^L \left( \prod_{\ell = 0}^{h - 1} \left| f_\ell \left( T_\ell^{n_j^{(\ell)}} x \right) \right| \right) \left| f_h \left( T_h^{n_j^{(h)}} x \right) - f_h \left( T_h^{n_j^{(h)}} x_0 \right) \right| \left( \left| \prod_{\ell = h + 1}^L f_\ell \left( T_\ell^{n_j^{(\ell)}} x_0 \right) \right| \right) \\
			\leq	& \sum_{h = 0}^L M^h \frac{\epsilon}{2 (L + 1) M^L} M^{L - h} \\
			& = \epsilon / 2 .
		\end{align*}
		On the other hand, if $j \in B_k$, then
		\begin{align*}
			& \left| \alpha_{C_k} \left( \left( f_0(x_0) \prod_{\ell = 1}^\ell T_\ell^{n_j^{(\ell)}} f_\ell(x_0) \right) - f_0 \prod_{\ell = 1}^\ell T_\ell^{n_j^{(\ell)}} f_\ell \right) \right| \\
			\leq	& \left\| \left( f_0(x_0) \prod_{\ell = 1}^\ell T_\ell^{n_j^{(\ell)}} f_\ell(x_0) \right) - f_0 \prod_{\ell = 1}^\ell T_\ell^{n_j^{(\ell)}} f_\ell \right\| \\
			\leq & (2M)^{L + 1} \\
			= & 2^{L + 1} M^{L + 1} .
		\end{align*}
		
		Now, choose $K \in \mathbb{N}$ such that if $k \geq K$, then
		$$\frac{ \left| \left\{ 0, 1, \ldots, k - 1 \right\} \setminus A_k \right| }{k} \leq \frac{\epsilon}{2^{L + 2} M^{L + 1}} .$$
		Then for all $k \geq K$, we have
		\begin{align*}
			& \left| \left( \sum_{j = 0}^{k - 1} f_0(x_0) \prod_{\ell = 1}^\ell T_\ell^{n_j^{(\ell)}} f_\ell(x_0) \right) - \alpha_{C_k} \left( \sum_{j = 0}^{k - 1} f_0 \prod_{\ell = 1}^\ell T_\ell^{n_j^{(\ell)}} f_\ell \right) \right| \\
			\leq	& \frac{1}{k} \sum_{j = 0}^{k - 1} \left| \left( f_0(x_0) \prod_{\ell = 1}^\ell T_\ell^{n_j^{(\ell)}} f_\ell(x_0) \right) - \alpha_{C_k} \left( f_0 \prod_{\ell = 1}^\ell T_\ell^{n_j^{(\ell)}} f_\ell \right) \right| \\
			=	& \frac{1}{k} \sum_{j = 0}^{k - 1} \left| \alpha_{C_k} \left( \left( f_0(x_0) \prod_{\ell = 1}^\ell T_\ell^{n_j^{(\ell)}} f_\ell(x_0) \right) - f_0 \prod_{\ell = 1}^\ell T_\ell^{n_j^{(\ell)}} f_\ell \right) \right| \\
			=	& \frac{1}{k} \left[ \sum_{j \in A_k} \left| \alpha_{C_k} \left( \left( f_0(x_0) \prod_{\ell = 1}^\ell T_\ell^{n_j^{(\ell)}} f_\ell(x_0) \right) - f_0 \prod_{\ell = 1}^\ell T_\ell^{n_j^{(\ell)}} f_\ell \right) \right| \right] \\
			& + \frac{1}{k} \left[ \sum_{j \in \{0, 1, \ldots, k - 1\} \setminus A_k} \left| \alpha_{C_k} \left( \left( f_0(x_0) \prod_{\ell = 1}^\ell T_\ell^{n_j^{(\ell)}} f_\ell(x_0) \right) - f_0 \prod_{\ell = 1}^\ell T_\ell^{n_j^{(\ell)}} f_\ell \right) \right| \right] \\
			\leq	& \frac{|A_k|}{k} \frac{\epsilon}{2} + \frac{ \left| \left\{ 0, 1, \ldots, k - 1 \right\} \setminus A_k \right| }{k} 2^{L + 1} M^{L + 1} . \\
			\leq	& \frac{\epsilon}{2} + \frac{\epsilon}{2^{L + 2} M^{L + 1}} 2^{L + 1} M^{L + 1} \\
			=	& \epsilon .
		\end{align*}
		Therefore
		$$
		\lim_{k \to \infty} \left| \left( \sum_{j = 0}^{k - 1} f_0(x_0) \prod_{\ell = 1}^\ell T_\ell^{n_j^{(\ell)}} f_\ell(x_0) \right) - \alpha_{C_k} \left( \sum_{j = 0}^{k - 1} f_0 \prod_{\ell = 1}^\ell T_\ell^{n_j^{(\ell)}} f_\ell \right) \right| = 0 .
		$$
	\end{proof}
	
	Theorem \ref{Pointwise reduction for multiple ergodic averages} can be used to convert pointwise convergence results for nonconventional ergodic averages into convergence results for random temporo-spatial averages, as shown by the following result.
	
	\begin{Cor}\label{Random multiple ergodic averages}
		Let $(X, p)$ be a compact pseudometric space, and let $T_1, \ldots, T_L$ be a family of homeomorphisms $T_\ell : X \to X$. Let $\mu$ be a regular Borel probability measure on $X$ invariant under each $T_\ell$. Let $\left( n_j^{(1)} \right)_{j = 1}^\infty , \ldots, \left( n_j^{(L)} \right)_{j = 1}^\infty$ be sequences of integers.
		
		For each point $x \in X$, let $(C_k(x))_{k = 1}^\infty$ be a sequence of measurable subsets of $X$ with positive measure for which $x \in C_k(x)$ and suppose that for each $\ell = 1, \ldots, \ell$, and every $\delta \in (0, \infty)$, we have that
		$$\frac{ \left| \left\{ j \in \{0, 1, \ldots, k - 1\} : \operatorname{diam} \left( T_\ell^{n_j^{(\ell)}} C_k(x) \right) \geq \delta \right\} \right| }{k} \to 0 . $$
		
		Let $f_0, f_1, \ldots, f_L \in C(X)$, and suppose that $f^* \in L^\infty(X, \mu)$ such that
		\begin{align*}
			\lim_{k \to \infty} \frac{1}{k} \sum_{j = 0}^{k - 1} f_0(x) \prod_{\ell = 1}^\ell T_\ell^{n_j^{(\ell)}} f_\ell (x)	& = f^*(x)
		\end{align*}
		for almost all $x \in X$.
		Then
		\begin{align*}
			\lim_{k \to \infty} \alpha_{C_k(x)} \left( \frac{1}{k} \sum_{j = 0}^{k - 1} f_0 \prod_{\ell = 1}^\ell T_\ell^{n_j^{(\ell)}} f_\ell \right)	& = f^*(x)
		\end{align*}
		for almost all $x \in X$.
	\end{Cor}
	
	\begin{proof}
		Set
		$$E = \left\{ x \in X : \lim_{k \to \infty} \frac{1}{k} \sum_{j = 0}^{k - 1} f_0(x) \prod_{\ell = 1}^\ell T_\ell^{n_j^{(\ell)}} f_\ell (x) = f^*(x) \right\} .$$
		If $x \in E$, then Theorem \ref{Pointwise reduction for multiple ergodic averages} tells us that
		$$
		\lim_{k \to \infty} \alpha_{C_k(x)} \left( \frac{1}{k} \sum_{j = 0}^{k - 1} f_0 \prod_{\ell = 1}^\ell T_\ell^{n_j^{(\ell)}} f_\ell \right) = f^*(x) .
		$$
	\end{proof}
	
	\section{Weighed temporo-spatial differentiation theorems}\label{Weighed temporo-spatial differentiation theorems}
	
	For the duration of this section, we narrow our attention to the case where $G = \mathbb{Z}$, and introduce a generalized form of a temporo-spatial differentiation problem. We also adopt the common notation that the action of the integer $n \in \mathbb{Z}$ be written as $T^n$. Let $(X, \mu)$ consist of a compact pseudometrizable space $X$ endowed with a Borel probability measure $\mu$, and let $T : X \to X$ be a homeomorphism. A \emph{weight} on $X$ is a measurable function $X \to \mathbb{T}$, where $\mathbb{T} = \left\{ z \in \mathbb{C} : |z| = 1 \right\}$. For convenience, write
	$$\operatorname{Avg}_{F}^\xi f : = \frac{1}{|F|} \sum_{j \in F} \xi^j \cdot \left( f \circ T^j \right) ,$$
	where $F$ is a finite nonempty subset of $\mathbb{Z}$. Let $(C_k)_{k = 1}^\infty$ be a sequence of measurable subsets of $X$ with $\mu(C_k) > 0$ for all $k \in \mathbb{N}$, and let $f \in L^\infty (X, \mu)$. What can be said of the limiting behavior of the sequence
	$$\alpha_{C_k} \left( \operatorname{Avg}_{F_k}^\xi f \right)_{k = 1}^\infty ?$$
	Moreover, suppose $\Xi$ is some family of measurable functions $X \to \mathbb{T}$. What can be said about the limiting behavior of the sequence 
	$$\alpha_{C_k} \left( \operatorname{Avg}_{F_k}^\xi f \right)_{k = 1}^\infty$$
	for all $\xi \in \Xi$?
	
	We consider this problem in analogy with a classical problem of pointwise weighted temporal averages.
	
	\begin{W-W}\label{W-W}
		Let $(X, \mu)$ be a standard probability space, and let $T$ be an automorphism of the probability space $(X, \mu)$. Set $[k] = \left\{0, 1, \ldots, k - 1\right\} \subseteq \mathbb{Z}$. Then for each $f \in L^1(\mu)$ exists a set $X_f \subseteq X$ of full measure such that for all $x \in X_f$, and all $\theta \in \mathbb{T}$, the sequence
		$$\left( \operatorname{Avg}_{[k]}^\theta f (x) \right)_{k = 1}^\infty$$
		converges, where we identify the unimodular complex number $\theta$ with the constant function $x \mapsto \theta$ on $X$.
	\end{W-W}
	
	The first alleged proof of the Wiener-Wintner Theorem was presented in \cite{WWOG}, but the argument presented was found to be incorrect. However, several proofs of the result have been presented since then. See \cite[Chapter 2]{AssaniWW} for a discussion of several different approaches to the result.
	
	As in Section \ref{Temporo-spatial differentiation theorems around sets of rapidly vanishing diameter}, we present a very general result that allows us to reduce certain temporo-spatial problems to certain pointwise temporal problems. Afterwards, we provide specific examples of this reduction. Before we can prove Lemma \ref{Non-Holder pointwise and temporo-spatial, weighted version}, we introduce some terminology and prove an elementary technical lemma.
	
	\begin{Def}
		Let $\xi : (X, p) \to \mathbb{C}$ be a complex-valued function on a pseudometric space $(X, p)$. A \emph{modulus of uniform continuity for $\xi$} is a function $\Delta : (0, 1) \to (0, \infty)$ such that
		\begin{align*}
			\forall \epsilon \in (0, 1) \;	& \forall x_1, x_2 \in X	& \left[ (p(x_1, x_2) \leq \Delta(\epsilon)) \Rightarrow |\xi(x_1) - \xi(x_2)| \leq \epsilon \right] .
		\end{align*}
		Given a family $\Xi$ of functions $(X, p) \to \mathbb{C}$, we call a function $\Delta : (0, 1) \to (0, \infty)$ a \emph{modulus of uniform equicontinuity} for $\Xi$ if $\Delta$ is a modulus of uniform continuity for all $\xi \in \Xi$.
	\end{Def}
	
	A function $\xi$ is of course uniformly continuous if and only if it admits a modulus of uniform continuity, and a family $\Xi$ is uniformly equicontinuous if and only if it admits a modulus of uniform equicontinuity. Note we make no assumption that a modulus of uniform continuity or modulus of uniform equicontinuity is the ``best possible" choice. For example, if $\Xi = \{1\}$ consists solely of the constant function $1$, then \emph{any} map $(0, 1) \to (0, 1)$ would be both a modulus of uniform continuity for $1$ and a modulus of uniform equicontinuity for $\Xi$.
	
	\begin{Lem}
		Let $(X, \rho)$ be a compact pseudometric space, and let $\Xi$ be a uniformly equicontinuous family of functions $(X, p) \to \mathbb{T}$ with modulus of uniform equicontinuity $\Delta$. Then $\Xi^j = \left\{ \xi^j : \xi \in \Xi \right\}$ is a uniformly equicontinuous family for all $j \in \mathbb{Z}$, and if $j \neq 0$, then $\epsilon \mapsto \Delta\left( \epsilon / |j| \right)$ is a modulus of uniform equicontinuity for $\Xi^j$.
	\end{Lem}
	
	\begin{proof}
		If $j = 0$, then $\Xi^j = \left\{1\right\}$, which is trivially uniformly equicontinuous, and in fact any map $(0, 1) \to (0, 1)$ whatsoever will be a modulus of uniform equicontinuity for $\Xi^0$. Now assume that $j \neq 0$.
		
		We prove this first for $j \in \mathbb{N}$, i.e. $j = |j| > 0$. Let $x_1, x_2 \in X , \xi \in \Xi$. We set up a telescoping sum
		\begin{align*}
			\left| \xi^j(x_1) - \xi^j(x_2) \right|	& = \left| \left( \xi(x_1) - \xi(x_2) \right) \sum_{p = 0}^{j - 1} \xi^p(x_1) \xi^{j - p - 1}(x_2) \right| \\
			& = \left| \xi(x_1) - \xi(x_2) \right| \cdot \left| \sum_{p = 0}^{j - 1} \xi^p(x_1) \xi^{j - p - 1}(x_2) \right| \\
			& \leq \left| \xi(x_1) - \xi(x_2) \right| \cdot \sum_{j = 0}^{p - 1} \left| \xi^p(x_1) \xi^{j - p - 1}(x_2) \right| \\
			& = \left| \xi(x_1) - \xi(x_2) \right| \cdot j .
		\end{align*}
		
		Now, in the case where $j < 0$, i.e. $j = - |j|$, we observe that $\Xi^j = \left(\Xi^{|j|}\right)^{-1} = \left\{ \overline{\zeta} : \zeta \in \Xi^j \right\}$, and conjugation is an isometry.
	\end{proof}
	
	\begin{Lem}\label{Non-Holder pointwise and temporo-spatial, weighted version}
		Let $(X, p)$ be a compact pseudometric space, and let $T : X \curvearrowright X$ be a homeomorphism of $X$. Let $\mu$ be a regular Borel probability measure on $X$. Fix a point $x_0 \in X$, and for each $n \in \mathbb{Z}, r \in (0, \infty)$, let $D_{x_0}(j, r)$ be the value
		$$D_{x_0}(j, r) = \sup \left\{ p \left( T^j x_0, T^j x \right) : x \in X , p(x_0, x) \leq r \right\} .$$
		
		Let $(F_k)_{k = 1}^\infty$ be a sequence of finite nonempty subsets of $\mathbb{Z}$. Let $\Xi$ be a uniformly equicontinuous family of continuous functions $X \to \mathbb{T}$, and for each $j \in \mathbb{Z}$, let $\Delta^j$ be a modulus of uniform equicontinuity for $\Xi^j$. Let $(C_k)_{k = 1}^\infty$ be a sequence of measurable subsets of $X$ such that $\mu(C_k) > 0$ and $x_0 \in C_k$ for all $k \in \mathbb{N}$. Suppose that for every $\delta > 0 , \epsilon > 0$, we have
		\begin{align*}
			\lim_{k \to \infty} \frac{\left| \left\{ j \in F_k : D_{x_0}(g, \operatorname{diam}(C_k)) > \delta \right\} \right|}{|F_k|}	& = 0 , \\
			\lim_{k \to \infty}	\frac{\left| \left\{ j \in F_k : \operatorname{diam} \left( C_k \right) > \Delta^j(\epsilon) \right\} \right| }{|F_k|}	& = 0 .
		\end{align*}
		Let $f \in C(X)$. Finally, suppose there exists a constant $\lambda > 0$ such that $\mu\left( T^j C_k \right) \leq \lambda \mu(C_k)$ for all $j \in \mathbb{N}$. Then for all $\xi \in \Xi$, we have
		$$
		\lim_{k \to \infty} \left| \left( \operatorname{Avg}_{F_k}^\xi f \right) (x_0) - \alpha_{C_k} \left( \operatorname{Avg}_{F_k}^\xi f \right) \right| = 0 ,
		$$
		and the convergence is uniform in $\xi \in \Xi$.
	\end{Lem}
	
	\begin{proof}
		Our proof of this result is similar in structure to our proof of Lemma \ref{Non-Holder pointwise and temporo-spatial}, but with the added wrinkle of accounting for how the weight affects our averages.
		
		Fix $\epsilon > 0$, and let $\xi \in \Xi$. Since $f$ is uniformly continuous, there exists $\delta_1 > 0$ such that if $y_1, y_2 \in X$, and $p(y_1, y_2) \leq \delta$, then $|f(y_1) - f(y_2)| \leq \frac{\epsilon}{4 \lambda \max \{ 1 , \|f\|_u \}}$, where $\| \cdot \|_u$ denotes the uniform norm on $C(X)$. Write
		\begin{align*}
			A_k	& = \left\{ j \in F_k : D_{x_0}(j, \operatorname{diam}(C_k)) \leq \delta , \; \operatorname{diam} \left( T^j C_k \right) \leq \Delta^j\left(\frac{\epsilon}{4 \lambda \max \{ 1 , \|f\|_u \}}\right) \right\} , \\
			B_k	& = F_k \setminus A_k .
		\end{align*}
		Our hypothesis tells us that $|A_k| / |F_k| \to 1 , |B_k| / |F_k| \to 0$. We estimate
		\begin{align*}
			& \left| \left( \operatorname{Avg}_{F_k}^\xi f \right)(x_0) - \alpha_{C_k} \left( \operatorname{Avg}_{F_k}^\xi \right) \right| \\
			=	& \left| \alpha_{C_k} \left( \left( \operatorname{Avg}_{F_k}^\xi f \right)(x_0) - \operatorname{Avg}_{F_k}^\xi f \right) \right| \\
			=	& \left| \alpha_{C_k} \left( \frac{1}{|F_k|} \sum_{j \in F_k} \left[ \xi(x_0)^j f \left( T^j x_0 \right) - \xi^j \left( f \circ T^j \right) \right] \right) \right| \\
			\leq & \left| \alpha_{C_k} \left( \frac{1}{|F_k|} \sum_{j \in A_k} \left[ \xi(x_0)^j f \left( T^j x_0 \right) - \xi^j \left( f \circ T^j \right) \right] \right) \right| \\
			& + \left| \alpha_{C_k} \left( \frac{1}{|F_k|} \sum_{j \in B_k} \left[ \xi(x_0)^j f \left( T^j x_0 \right) - \xi^j \left( f \circ T^j \right) \right] \right) \right|
		\end{align*}
		We estimate these two terms separately, starting with the first.
		
		If $j \in A_k$, then
		\begin{align*}
			& \left| \alpha_{C_k} \left( \xi(x_0)^j f \left( T^j x_0 \right) - \xi^j \left( f \circ T^j \right) \right) \right| \\
			=	& \left| \frac{1}{\mu(C_k)} \int_{C_k} \left( \xi^j(x_0) f \left( T^j x_0 \right) - \xi^j(x) f \left( T^j x \right) \right) \mathrm{d} \mu(x) \right| \\
			\leq	& \frac{1}{\mu(C_k)} \int_{C_k} \left| \xi^j(x_0) f \left( T^j x_0 \right) - \xi^j(x) f \left( T^j x \right) \right| \mathrm{d} \mu(x) \\
			=	& \frac{1}{\mu(C_k)} \int_{T^j C_k} \left| \xi^j \left( x_0 \right) f \left( T^j x_0 \right) - \xi^j \left( T^{-j} y \right) f \left( y \right) \right| \mathrm{d} \mu(y) \\
			\leq	& \frac{\lambda}{\mu \left( T^j C_k \right)} \int_{T^j C_k} \left| \xi^j \left( x_0 \right) f \left( T^j x_0 \right) - \xi^j \left( T^{-j} y \right) f \left( y \right) \right| \mathrm{d} \mu(y) \\
			\leq	& \frac{\lambda}{\mu \left( T^j C_k \right)} \int_{T^j C_k} \left| \xi^j \left( x_0 \right) f \left( T^j x_0 \right) - \xi^j(x_0)f \left( y \right) \right| \mathrm{d} \mu(y) \\
			& + \frac{\lambda}{\mu \left( T^j C_k \right)} \int_{T^j C_k} \left| \xi^j \left( x_0 \right) f \left( y \right) - \xi^j \left( T^{-j} y \right) f \left( y \right) \right| \mathrm{d} \mu(y) \\
			=	& \frac{\lambda}{\mu \left( T^j C_k \right)} \int_{T^j C_k} \left| \xi^j(x_0) \right| \left| f \left( T^j x_0 \right) - f \left( y \right) \right| \mathrm{d} \mu(y) \\
			& + \frac{\lambda}{\mu \left( T^j C_k \right)} \int_{T^j C_k} \left| \xi^j \left( x_0 \right) - \xi^j \left( T^{-j} y \right) \right| \cdot \left| f(y) \right| \mathrm{d} \mu(y) \\
			\leq	& \frac{\lambda}{\mu \left( T^j C_k \right)} \int_{T^j C_k} \left| f \left( T^j x_0 \right) - f \left( y \right) \right| \mathrm{d} \mu(y) \\
			& + \frac{\lambda}{\mu \left( T^j C_k \right)} \int_{T^j C_k} \left| \xi^j \left( x_0 \right) - \xi^j \left( T^{-j} y \right) \right| \cdot \left\| f \right\|_\infty \mathrm{d} \mu(y) .
		\end{align*}
		First, if $y \in T^j C_k$, and $D_{x_0} \left( j , \operatorname{diam}(C_k) \leq \delta \right)$, then $p \left( T^j x_0, y \right) \leq \delta$, meaning that
		$$
		\frac{\lambda}{\mu \left( T^j C_k \right)} \int_{T^j C_k} \left| f \left( T^j x_0 \right) - f \left( y \right) \right| \mathrm{d} \mu(y) \leq \frac{\epsilon}{4} ,
		$$
		and $y \in T^j C_k \Rightarrow T^{-j} y \in C_k$, meaning that $p \left( x_0, T^{-j} y \right) \leq \operatorname{diam}(C_k) \leq \Delta^j \left( \frac{\epsilon}{4 \lambda \max \left\{ 1 , \|f\|_u \right\}} \right)$, so
		$$
		\frac{\lambda}{\mu \left( T^j C_k \right)} \int_{T^j C_k} \left| \xi^j \left( x_0 \right) - \xi^j \left( T^{-j} y \right) \right| \cdot \left\| f \right\|_\infty \mathrm{d} \mu(y) \leq \frac{\epsilon}{4} .
		$$
		Thus
		$$
		\left| \alpha_{C_k} \left( \frac{1}{|F_k|} \sum_{j \in A_k} \left[ \xi(x_0)^j f \left( T^j x_0 \right) - \xi^j \left( f \circ T^j \right) \right] \right) \right| \leq \frac{\epsilon}{2} .
		$$
		
		Suppose now that $j \in B_k$. By a computation similar to the one performed for the case where $j \in A_k$, we get
		\begin{align*}
			& \left|\xi(x_0)^j f \left( T^j x_0 \right) - \xi^j \left( f \circ T^j \right)\right|\\
			\leq	& \frac{\lambda}{\mu \left( T^j C_k \right)} \int_{T^j C_k} \left| \xi^j \left( x_0 \right) f \left( T^j x_0 \right) - \xi^j \left( T^{-j} y \right) f \left( y \right) \right| \mathrm{d} \mu(y) \\
			\leq	& \frac{\lambda}{\mu \left( T^j C_k \right)} \int_{T^j C_k} \left( \left| \xi^j \left( x_0 \right) f \left( T^j x_0 \right) \right| + \left| \xi^j \left( T^{-j} y \right) f \left( y \right) \right| \right) \mathrm{d} \mu(y) \\
			\leq	& \frac{\lambda}{\mu \left( T^j C_k \right)} \mu \left( T^j C_k \right) \left( 2 \|f\|_u \right) \\
			=	& 2 \lambda \|f\|_u .
		\end{align*}
		Choose $K \in \mathbb{N}$ such that if $k \geq K$, then $\frac{|B_k|}{|F_k|} \leq \frac{\epsilon}{4 \lambda \max \left\{ 1 , \|f\|_u \right\}}$. Then if $k \geq K$, we have
		\begin{align*}
			& \left| \left( \operatorname{Avg}_{F_k}^\xi f \right)(x_0) - \alpha_{C_k} \left( \operatorname{Avg}_{F_k}^\xi \right) \right| \\
			\leq & \left| \alpha_{C_k} \left( \frac{1}{|F_k|} \sum_{j \in A_k} \left[ \xi(x_0)^j f \left( T^j x_0 \right) - \xi^j \left( f \circ T^j \right) \right] \right) \right| \\
			& + \left| \alpha_{C_k} \left( \frac{1}{|F_k|} \sum_{j \in B_k} \left[ \xi(x_0)^j f \left( T^j x_0 \right) - \xi^j \left( f \circ T^j \right) \right] \right) \right| \\
			\leq	& \frac{\epsilon}{2} + \frac{\epsilon}{2} \\
			& = \epsilon .
		\end{align*}
		We note that our estimates on $K$ were independent of our choice of $\xi \in \Xi$, meaning the convergence is uniform in $\xi$.
	\end{proof}
	
	With this in mind, we can state the following.
	
	\begin{Thm}
		Let $(X, p)$ be a compact pseudometric space, and let $T : X \curvearrowright X$ be a homeomorphism of $X$. Let $\mu$ be a regular Borel probability measure on $X$. Suppose there exist functions $H , L : \mathbb{Z} \to (0, \infty)$ such that
		\begin{align*}
			p \left( T^g x , T^g y \right)	& \leq L(j) \cdot p(x, y)^{H(j)}	& \left( \forall j \in \mathbb{Z} , x \in X , y \in X \right) .
		\end{align*}
		
		Let $(F_k)_{k = 1}^\infty$ be a sequence of finite nonempty subsets of $\mathbb{Z}$. Let $\Xi$ be a uniformly equicontinuous family of continuous functions $X \to \mathbb{T}$, and for each $j \in \mathbb{Z}$, let $\Delta^j$ be a modulus of uniform equicontinuity for $\Xi^j$. Let $(C_k)_{k = 1}^\infty$ be a sequence of measurable subsets of $X$ such that $\mu(C_k) > 0$ and $x_0 \in C_k$ for all $k \in \mathbb{N}$. Suppose that for every $\delta > 0 , \epsilon > 0$, we have
		\begin{align*}
			\lim_{k \to \infty} \frac{\left| \left\{ j \in F_k : L(j) \cdot \operatorname{diam}(C_k)^{H(j)} > \delta \right\} \right|}{|F_k|}	& = 0 , \\
			\lim_{k \to \infty}	\frac{\left| \left\{ j \in F_k : \operatorname{diam} \left( C_k \right) > \Delta^j(\epsilon) \right\} \right| }{|F_k|}	& = 0 .
		\end{align*}
		Let $x_0 \in X$ be a point in $X$, and let $f : X \to \mathbb{C}$ be a uniformly bounded continuous function Finally, suppose there exists a constant $\lambda > 0$ such that $\mu\left( T^j C_k \right) \leq \lambda \mu(C_k)$ for all $j \in \mathbb{N}$. Then for all $\xi \in \Xi$, we have
		$$
		\lim_{k \to \infty} \left| \left( \operatorname{Avg}_{F_k}^\xi f \right) (x_0) - \alpha_{C_k} \left( \operatorname{Avg}_{F_k}^\xi f \right) \right| = 0 ,
		$$
		and the convergence is uniform in $\xi \in \Xi$.
	\end{Thm}
	
	\begin{proof}
		We have the bound $D_{x_0} (j, r) \leq L(j) \cdot r^{H(j)}$. We can thus apply Lemma \ref{Non-Holder pointwise and temporo-spatial, weighted version}.
	\end{proof}
	
	\begin{Cor}\label{W-W stds}
		Let $(X, \rho)$ be a compact metric space, and let $T : X \to X$ be a homeomorphism. Let $\mu$ be a Borel probability measure on $X$ that's $T$-invariant. Let $f \in C(X)$. For each $x \in X$, let $C_k(x)$ be a measurable subset of $X$ with positive measure such that
		\begin{align*}
			\lim_{k \to \infty} \frac{\left| \left\{ j \in [k] : D_x(j, \operatorname{diam}(C_k(x))) > \delta \right\} \right|}{k}	& = 0 , \\
			\lim_{k \to \infty}	\frac{\left| \left\{ j \in [k] : \operatorname{diam} \left( C_k(x) \right) > \Delta^j(\epsilon) \right\} \right| }{k}	& = 0 .
		\end{align*}
		Then for every $f \in C(X)$ exists a set $X_f \subseteq X$ of full measure such that for all $x \in X_f$, and all $\theta \in \mathbb{T}$, the sequence
		$$\left( \alpha_{C_k(x)} \left( \operatorname{Avg}_{[k]}^\theta \right) \right)_{k = 1}^\infty $$
		converges.
	\end{Cor}
	
	\begin{proof}
		By the Wiener-Wintner pointwise ergodic theorem, there exists a set $X_f \subseteq X$ of full measure such that for all $x \in X_f$, and all $\theta \in \mathbb{T}$, the sequence
		$$\left( \operatorname{Avg}_{[k]}^\theta f (x) \right)_{k = 1}^\infty$$
		converges. By Lemma \ref{Non-Holder pointwise and temporo-spatial, weighted version}, it follows that $\lim_{k \to \infty} \left| \alpha_{C_k(x)} \left( \operatorname{Avg}_{[k]}^\theta f \right) - \left( \operatorname{Avg}_{[k]}^\theta f (x) \right) \right| = 0$. Thus the sequence converges.
	\end{proof}
	
	We now consider a different class of weighting sequences, where we choose our weights to be constant functions, but loosen our assumptions about boundedness. Given a sequence $\mathbf{F} = (F_k)_{k = 1}^\infty$ of finite subsets of $\mathbb{Z}$, set
	$$M^{\mathbf{F}} : =\left\{ (a_k)_{k \in \mathbb{Z}} \in \mathbb{C}^\mathbb{Z} : \sup_{\ell \in \mathbb{N}} \frac{1}{|F_\ell|} \sum_{j \in F_\ell} |a_j| < \infty \right\}.$$
	We also introduce the notation
	$$\operatorname{Avg}_{F}^a : = \frac{1}{|F|} \sum_{j \in F} a_j f \circ T^j ,$$
	where $F$ is a finite subset of $\mathbb{Z}$.
	
	Our next result establishes that under a rapidly decaying diameter condition, temporo-spatial differentiations involving weighted ergodic means for continuous functions can be reduced to pointwise temporal averages. The twist here is that the diameter decay condition also hinges on the weighting sequence.
	
	\begin{Prop}\label{Non-Holder pointwise and temporo-spatial, weighted sequences of numbers}
		Let $(X, p)$ be a compact pseudometric space, and let $T : X \curvearrowright X$ be a homeomorphism of $X$. Let $\mu$ be a regular Borel probability measure on $X$. Fix a point $x_0 \in X$, and for each $n \in \mathbb{Z}, r \in (0, \infty)$, let $D_{x_0}(j, r)$ be the value
		$$D_{x_0}(j, r) = \sup \left\{ p \left( T^j x_0, T^j x \right) : x \in X , p(x_0, x) \leq r \right\} .$$
		
		Let $\mathbf{F} = (F_k)_{k = 1}^\infty$ be a sequence of finite nonempty subsets of $\mathbb{Z}$. Let $(a_k)_{k = 0}^\infty \in M^{\mathbf{F}}$, and let $(C_k)_{k = 1}^\infty$ be a sequence of measurable subsets of $X$ such that $\mu(C_k) > 0$ and $x_0 \in C_k$ for all $k \in \mathbb{N}$. Suppose that for every $\delta > 0$, we have
		\begin{align*}
			\lim_{k \to \infty} \frac{1}{|F_k|} \sum_{j \in F_k , \; D_{x_0}(j, \operatorname{diam}(C_k)) > \delta} |a_j|	& = 0 ,
		\end{align*}
		Let $f : X \to \mathbb{C}$ be a continuous function. Finally, suppose there exists a constant $\lambda > 0$ such that $\mu\left( T^j C_k \right) \leq \lambda \mu(C_k)$ for all $j \in \mathbb{N}$. Then we have
		$$
		\lim_{k \to \infty} \left| \left( \operatorname{Avg}_{F_k}^a f \right) (x_0) - \alpha_{C_k} \left( \operatorname{Avg}_{F_k}^a f \right) \right| = 0 .
		$$
	\end{Prop}
	
	\begin{proof}
		Fix $\epsilon > 0$. Appealing to the uniform continuity and boundedness of $f$, choose $\delta > 0$ such that
		$$p(y_1, y_2) \leq \delta \Rightarrow |f(y_1) - f(y_2)| \leq \epsilon .$$
		Set
		\begin{align*}
			A_k	& = \left\{ j \in F_k : D_{x_0} (j, \operatorname{diam}(C_k)) \leq \delta \right\} , \\
			B_k	& = \left\{ j \in F_k : D_{x_0} (j, \operatorname{diam}(C_k)) > \delta \right\} .
		\end{align*}
		Then $\frac{|A_k|}{|F_k|} \leq 1$, and $\frac{1}{|F_k|} \sum_{j \in B_k} |a_j| \to 0$. Using a calculation similar to that used in our proof of Lemma \ref{Non-Holder pointwise and temporo-spatial, weighted version}, we get
		\begin{align*}
			\\
			\left| \left( \operatorname{Avg}_{F_k}^a f \right)(x_0) - \alpha_{C_k} \left( \operatorname{Avg}_{F_k}^a \right) \right|	\leq & \left| \alpha_{C_k} \left( \frac{1}{|F_k|} \sum_{j \in A_k} \left[ a_j f \left( T^j x_0 \right) - a_j \left( f \circ T^j \right) \right] \right) \right| \\
			& + \left| \alpha_{C_k} \left( \frac{1}{|F_k|} \sum_{j \in B_k} \left[ a_j f \left( T^j x_0 \right) - a_j \left( f \circ T^j \right) \right] \right) \right| .
		\end{align*}
		As before, we'll estimate these two terms separately. Many of the calculations done here are quite similar to those used in our proof of Lemma \ref{Non-Holder pointwise and temporo-spatial, weighted version}, so we will be terser in our presentation here.
		
		First, suppose $j \in A_k$. Then
		\begin{align*}
			\left| \alpha_{C_k} \left( a_j f \left( T^j x_0 \right) - a_j \left( f \circ T^j \right) \right) \right|	& = \frac{|a_j|}{\mu(C_k)} \left| \int_{C_k} \left( f\left( T^j x_0 \right) - f \left( T^j x \right) \right) \mathrm{d} \mu(x) \right| \\
			& \leq \frac{|a_j|}{\mu(C_k)} \int_{C_k} \left| f\left( T^j x_0 \right) - f \left( T^j x \right) \right| \mathrm{d} \mu(x) \\
			& \leq \lambda |a_j| \frac{1}{\mu \left(T^j C_k \right)} \int_{T^j C_k} \left| f\left( T^j x_0 \right) - f \left( y \right) \right| \mathrm{d} \mu(y) \\
			& \leq \lambda |a_j| \epsilon .
		\end{align*}
		Therefore
		\begin{align*}
			& \left| \alpha_{C_k} \left( \frac{1}{|F_k|} \sum_{j \in A_k} \left[ a_j f \left( T^j x_0 \right) - a_j \left( f \circ T^j \right) \right] \right) \right| \\
			\leq	& \frac{1}{|F_k|} \sum_{j \in A_k} \left| \alpha_{C_k} \left( a_j f \left( T^j x_0 \right) - a_j \left( f \circ T^j \right) \right) \right| \\
			\leq	& \frac{1}{|F_k|} \sum_{j \in A_k} \lambda |a_j| \epsilon \\
			\leq	& \frac{1}{|F_k|} \sum_{j \in F_k} |a_j| \lambda \epsilon \\
			\leq	& \left( \sup_{\ell \in \mathbb{N}} |F_\ell|^{-1} \sum_{j \in F_\ell} |a_j| \right) \lambda \epsilon .
		\end{align*}
		
		Now, consider the case where $j \in B_k$. Then
		\begin{align*}
			\left| \alpha_{C_k} \left( a_j f \left( T^j x_0 \right) - a_j \left( f \circ T^j \right) \right) \right|	& \leq \lambda |a_j| \frac{1}{\mu \left(T^j C_k \right)} \int_{T^j C_k} \left| f\left( T^j x_0 \right) - f \left( y \right) \right| \mathrm{d} \mu(y) \\
			& \leq \lambda |a_j| \left( 2 \|f\|_u \right) .
		\end{align*}
		Choose $K \in \mathbb{N}$ such that if $k \geq K$, then $\frac{1}{|F_k|} \sum_{j \in B_k} |a_j| \leq \epsilon$. Then
		\begin{align*}
			\left| \alpha_{C_k} \left( \frac{1}{|F_k|} \sum_{j \in A_k} \left[ a_j f \left( T^j x_0 \right) - a_j \left( f \circ T^j \right) \right] \right) \right|	& \leq \frac{1}{|F_k|} \sum_{j \in B_k} \lambda |a_j| (2 \|f\|_u) \\
			& \leq 2 \lambda \|f\|_u \frac{1}{|F_k|} \sum_{j \in B_k} |a_j| \\
		\end{align*}
		
		Therefore, if $k \geq K$, we have
		\begin{align*}
			& \left| \left( \operatorname{Avg}_{F_k}^a f \right)(x_0) - \alpha_{C_k} \left( \operatorname{Avg}_{F_k}^a \right) \right| \\
			\leq & \left| \alpha_{C_k} \left( \frac{1}{|F_k|} \sum_{j \in A_k} \left[ a_j f \left( T^j x_0 \right) - a_j \left( f \circ T^j \right) \right] \right) \right| \\
			& + \left| \alpha_{C_k} \left( \frac{1}{|F_k|} \sum_{j \in B_k} \left[ a_j f \left( T^j x_0 \right) - a_j \left( f \circ T^j \right) \right] \right) \right| \\
			\leq	& \left( \sup_{\ell \in \mathbb{N}} |F_\ell|^{-1} \sum_{j \in F_\ell} |a_j| \right) \lambda \epsilon + 2 \lambda \|f\|_u \epsilon \\
			=	& \lambda \left( \left( \sup_{\ell \in \mathbb{N}} |F_\ell|^{-1} \sum_{j \in F_\ell} |a_j| \right) + 2 \|f\|_u \right) \epsilon .
		\end{align*}
		This coefficient on $\epsilon$ is independent of our choice of $k$, so we can conclude that
		$$\lim_{k \to \infty} \left| \left( \operatorname{Avg}_{F_k}^a f \right) (x_0) - \alpha_{C_k} \left( \operatorname{Avg}_{F_k}^a f \right) \right| = 0 .$$
	\end{proof}

\section*{Acknowledgments}

This paper is written as part of the author's graduate studies. He is grateful to his beneficent advisor, professor Idris Assani, for no shortage of helpful guidance.

An earlier version of this paper referred to ``tempero-spatial differentiations." Professor Mark Williams pointed out that the more correct portmanteau would be ``temporo-spatial." We thank Professor Williams for this observation.
	
	\bibliography{Bibliography}
\end{document}